\documentclass[11pt]{amsart} 
\usepackage{amsmath}
\usepackage{amssymb}
\usepackage{harpoon,mathabx}
\usepackage{enumerate}
\usepackage[dvips]{epsfig}
\usepackage{enumitem}
\usepackage{amsmath}
\usepackage{graphicx}
\usepackage{pgfplots}
\pgfplotsset{compat=1.11}
\usepackage{bm}
\usepackage{subfig,tikz}
\usetikzlibrary{decorations.markings}
\usepackage{wrapfig}
\usepackage{xspace}
\theoremstyle{plain}
\newtheorem{theorem}{Theorem}[section]
\newtheorem{lemma}[theorem]{Lemma}
\newtheorem{prop}[theorem]{Proposition}
\newtheorem{cor}[theorem]{Corollary}

\newtheorem{proposition}[theorem]{Proposition}
\newcommand{\norm}[1]{\left\lVert#1\right\rVert}
\newtheorem{remark}[theorem]{Remark}
\newtheorem{definition}[theorem]{Definition}

\numberwithin{equation}{section}

\allowdisplaybreaks

\begin{document}

\title[Positive Mass Theorem for Initial Data Sets with Corners]
{Positive mass theorem for initial data sets with corners along a hypersurface}
\author{Aghil Alaee}
\address{Aghil Alaee, Center of  Mathematical Sciences and Applications\\
Harvard University, USA}
\email{aghil.alaee@cmsa.fas.harvard.edu}
\author{Shing-Tung Yau}
\address{Shing-Tung Yau, Department of Mathematics\\
	Harvard University, USA}
\email{yau@math.harvard.edu}
\begin{abstract}
We prove positive mass theorem with angular momentum and charges for axially symmetric, simply connected, maximal, complete initial data sets with two ends, one designated asymptotically flat and the other either (Kaluza-Klein) asymptotically flat or asymptotically cylindrical, for 4-dimensional Einstein-Maxwell theory and $5$-dimensional minimal supergravity theory  which metrics fail to be $C^1$ and second fundamental forms and electromagnetic fields fail to be $C^0$ across an axially symmetric hypersurface $\Sigma$. Furthermore, we remove the completeness and simple connectivity assumptions in this result and prove it for manifold with boundary such that the mean curvature of the boundary is non-positive.
\end{abstract}

\maketitle
\section{Introduction}
In 2007, Dain proved a positive mass theorem with angular momentum $\mathcal{J}$ \cite{dain2008proof},
\begin{equation}\label{Dain}
m\geq \sqrt{|\mathcal{J}|},
\end{equation}for 3-dimensional, smooth, axially symmetric, simply connected, complete, maximal initial data sets for vacuum Einstein equations with two ends, one designated asymptotically flat and the other either asymptotically flat or asymptotically cylindrical. Moreover, he proved that the rigidity of inequality \eqref{Dain} holds for the canonical slice of the extreme Kerr spacetime. The main physical motivation of this inequality is by the standard picture of gravitational collapse through the final state conjecture and the weak cosmic censorship conjecture \cite{dain2012geometric}. From then, the main objective of research in this direction was to remove the unnecessary assumptions in the Dain inequality and extend it to different physical theories and higher dimensions. 

The conditions of axial symmetry, vacuum, and two ends are essential to have a non-zero conserved angular momentum. However, Chru\'sciel, Li, and Weinstein \cite{chrusciel2008mass1} extend the proof to initial data sets for multiple black holes with non-negative energy density condition. Moreover,  Dain, Khuri, Weinstein, and Yamada \cite{dain2013lower} replaced the vacuum energy flux condition with vanishing energy flux in the direction of axial symmetry. On the other hand, inequality \eqref{Dain} has also been generalized to the Einstein-Maxwell theory by Costa \cite{costa2010proof}, that is
\begin{equation}\label{in3d}
	m^2\geq \frac{Q^2+\sqrt{Q^4+4\mathcal{J}^2}}{2},
\end{equation}where $Q=\sqrt{Q_e^2+Q_b^2}$ such that $Q_e$ and $Q_b$ are the electric and magnetic charges, respectively. Moreover, the equality holds for canonical slice of the extreme Kerr-Newman spacetime. 

In \cite{schoen2013convexity}, Schoen and Zhou developed an alternative proof of (charged-) Dain inequality using convexity of reduced harmonic energy. Moreover, Zhou \cite{zhou2014} treated the near maximal initial data sets and Cha and Khuri \cite{cha2015deformations} extended the proof of inequality to non-maximal initial data sets assuming a system of equations admits a solution. Then Khuri and Wienstein \cite{khuri2016positive} proved these inequalities with optimal asymptotic decay conditions for multiple charge black hole initial data sets. 

In general, the topological censorship theorem says that the domain of outer communication of asymptotically flat black holes must be simply connected \cite{friedman1993topological,galloway1995topology}. But there is no topological restriction in the inside region of a horizon. This shows that the simple connectivity and two ends assumptions in above inequalities are only technical restriction and not physical features of black holes. Therefore, recently, Khuri, Bryden, and Sokolowsky \cite{bryden2018positive} showed that the strict inequality of \eqref{in3d} is true for initial data sets with minimal axially symmetric boundary $\partial M^3$ without simple connectivity assumption.

In higher dimensions, a generalization of Hawking topology theorem, by Galloway and Schoen \cite{galloway2006generalization}, states that the cross sections of the event horizon of black holes must be positive Yamabe type. In particular, in five dimensions the only admissible horizon topologies are $3$-sphere (or its quotient), $S^1\times S^2$, or connected sum of these cases. For higher dimensional black holes similar to 3-dimensional black holes, we need to impose additional axial symmetry to have well-defined conserved total angular momenta. In order to accommodate the desired amount of axial symmetry as well as asymptotic flatness, the dimension of spacetime must be restricted to five. Therefore, the first author, Khuri, and Kunduri \cite{alaee2017relating} proved a positive mass theorem with angular momentum and charges for black holes with $3$-sphere horizon topology in the minimal supergravity theory,
\begin{equation}
m\geq \frac{27\pi}{8}\frac{\left(\mathcal{J}_1+\mathcal{J}_2\right)^2}{\left(2m+\sqrt{3}|Q|\right)^2}+\sqrt{3}|Q|,
\end{equation}where $\mathcal{J}_i$ are angular momenta corresponding to $U(1)^2$-symmetry and $Q$ is an electric charge. Moreover, equality holds for canonical slice of the extreme charged Myers-Perry spacetime. In particular, if $Q=0$, we derive the inequality in the vacuum \cite{alaee2015proof}, that is
\begin{equation}
m^3\geq \frac{27\pi}{8}\left(|\mathcal{J}_1|+|\mathcal{J}_2|\right)^2,
\end{equation}where the rigidity is canonical slice of the extreme Myers-Perry spacetime. The inequality for black ring with horizon topology $S^1\times S^2$ is only proved in the Einstein theory because of lack of an explicit extreme black hole solution with two nonzero angular momenta in the minimal supergravity. In particular, if the topology of initial data set is $M^4=S^2\times B^2\#\mathbb{R}^4$ and $\mathcal{J}_1\geq \mathcal{J}_2$, where $\mathcal{J}_1$ is angular momentum in direction of $S^1$ and $\mathcal{J}_2$ is angular momentum in the direction of $S^2$, then the first author, Khuri, and Kunduri \cite{alaee2017mass} also proved the following inequality
	\begin{equation*}
m^3\geq \frac{27\pi}{4}|\mathcal{J}_2||\mathcal{J}_1-\mathcal{J}_2|.
\end{equation*}Moreover, if $\mathcal{J}_1> \mathcal{J}_2$, then equality holds if and only if the initial data set arise from the canonical slice of the extreme Pomeransky-Se\'nkov black ring spacetime. A remarkable feature of this result is that the manifold is not doubling of a Cauchy surface in the domain of outer communication. Moreover, the second end has asymptotic topology $S^1\times S^2$ with different decay conditions on the metric components which separate the Kaluza-Klein-asymptotically flat from asymptotically cylindrical.

The purpose of this paper is to prove all of the above inequalities without the assumption of smoothness. Similar to above inequalities, the proof of positive mass theorem by Schoen and the second author was for smooth initial data sets \cite{schoen1979proof,schoen1981proof}. However, Bray showed that the Riemannian positive mass theorem holds for smooth metric which are Lipschitz on minimal hypersurfaces \cite{bray2001proof}. Moreover, Miao \cite{miao2002positive} extend the Riemannian positive mass theorem for metrics with corners across a hypersurface $\Sigma$. On the other hand, in the proof of positivity of quasi-local masses by Shi and Tam \cite{shi2002positive} and Liu and the second author \cite{liu2006positivity}, an analogous result holds for spin Riemannian manifold. This result has been generalized for a class of Lipschitz metrics where the complement of some singular set $S$ of metrics has Minkowski dimension less than $n/2$ by Lee \cite{lee2013positive}. More recently, Lee and LeFloch \cite{lee2015positive}
proved a Riemannian positive mass theorem for spin manifolds which metrics have the Sobolev regularity $C^0\cap W^{1,n}$. For initial data sets with dominant energy condition,  Shibuya \cite{shibuya2018lorentzian} showed that the positive mass theorem still holds if it has a causal corner, i.e., dominant energy condition does not violate along corner. Moreover, the authors and Khuri \cite{alaee2019geometric} extended the Miao smoothing to a gluing of the Jang deformation of initial data sets with dominant energy condition and proved geometric inequalities for spacetime quasi-local masses. A similar smooth gluing of the Jang deformation developed in the proof of a localized spacetime Penrose inequality by the authors and Lesourd \cite{alaee2019localized}.

The present article is concerned with axially symmetric initial data sets with corners across an axially symmetric hypersurface $\Sigma$. We show that an initial data set can be deform to initial data set with $C^2$ metric, $C^1$ second fundamental form, and $C^1$ electromagnetic field such that it is also a solution of Maxwell equations, has vanishing energy flux in the direction of symmetries, and has non-negative energy density. Moreover, the angular momentum and charges are conserved along this deformation and its mass converges to the mass of the initial data set with corners. This is enough to implement the current results and achieve above inequalities for non-smooth initial data sets. Finally, we prove the rigidity cases for all of these inequalities.

\section{Statement of main results}
In order to state the main result we first discuss the appropriate setting and note that we only consider three and four dimensional initial data sets, that is $n=3,4$. An initial data set $(M^n,g,k,E,B)$ of 4-dimensional Einstein-Maxwell theory and 5-dimensional minimal supergravity theory are consists of a $n$-dimensional Riemannian manifold $M^n$, with metric $g$, second fundamental form $k$, electric field $E$, and magnetic $(n-2)$-form $B$ which is a solution of the following constraint equations.
\begin{equation}\label{Hcons}
16\pi\mu=R(g)-|k|^2_g+(\text{tr}_gk)^2-\frac{2}{(n-2)^2}|E|^2_g-\frac{2}{(n-2)^3}|B|^2_g,
\end{equation}
\begin{equation}\label{momeq}
8\pi J=\text{div}_g\left(k-(\text{tr}_gk) g\right)+\frac{2}{(n-2)^2}\star\left(B\wedge E\right),
\end{equation}
\begin{equation}\label{EMfield}
\text{div}_gE=\frac{n-3}{\sqrt{3}}\star\left(B\wedge B\right),\qquad \text{div}_gB=0,
\end{equation}where $\star$ is the Hodge star operator with respect to metric $g$,  $\mu$ is energy density, and $J$ is energy flux of non-electromagnetic fields. To have a conserved angular momentum, we need axially symmetric initial data set. In particular, we have the following definition.
\begin{definition}An initial data set $(M^n,g,k,E,B)$ is axially symmetric if there exists a $U(1)^{n-2}$ subgroup within the group of isometries of the Riemannian manifold $(M^n, g)$ so that  the Lie derivative of initial data set vanishes along axial symmetry, that is
	\begin{equation}
	\mathcal{L}_{\eta_{(l)}}g=\mathcal{L}_{\eta_{(l)}}k=\mathcal{L}_{\eta_{(l)}}E=\mathcal{L}_{\eta_{(l)}}B=0,
	\end{equation}where $\eta_{(l)}$, for $l=1,n-2$, are generators of $U(1)^{n-2}$ group.
\end{definition}The initial data set for an isolated system consist of a Riemannian manifold with asymptotically flat end which means there exists a sub-manifold $M_{\text{end}}\subset M^n$ diffeomorphic to $R^n\,\backslash B_r(0)$ such that in local coordinate on $M_{\text{end}}$ obtained from $R^n\,\backslash B_r(0)$, we have the following fall-off conditions.
\begin{equation}
g_{ij}=\delta_{ij}+o_{s}(r^{-\frac{n-2}{2}}),\qquad \partial g_{ij}\in L^2(M_{\text{end}}),\qquad k_{ij}=o_{s-1}(r^{-\frac{n}{2}}),
\end{equation}
\begin{equation}
B^i,E^i=o_{s-1}(r^{-\frac{n}{2}}),\qquad B_{ij}=o_{s-1}(r^{-\frac{n}{2}}),\qquad  \mu,J,J(\eta_{(l)})\in L^1(M_{\text{end}}),
\end{equation} for some $s\geq 5$. The assumption $s\geq 5$ is needed for existence of the Brill coordinate system \cite{chrusciel2008mass} in three dimensions. These fall-off conditions ensure that the ADM energy, angular momentum, and total electric and magnetic charges are well-defined.
The ADM energy of asymptotically flat initial data sets is defined by
\begin{equation}
m=\frac{1}{16\pi}\int_{S^{n-1}_{\infty}}\left(g_{ij,i}-g_{ii,j}\right)\nu^j
\end{equation}where $S^n_{\infty}$ is a coordinate sphere at infinity with unit outer normal $\nu$. The total angular momentum  of corresponding rotational symmetry $\eta_{(l)}$ is defined by 
\begin{equation}\label{AM}
\mathcal{J}_l=\frac{1}{8\pi}\int_{S^{n-1}_{\infty}}\left(k_{ij}-(\text{tr}_gk)g_{ij}\right)\nu^j\eta^i_{(l)}.
\end{equation}Moreover, the total electric and magnetic charges of initial data sets are defined by the following flux integrals at infinity
\begin{equation}\label{charges}
Q_e=\frac{1}{4(n-2)^2\pi}\int_{S^{n-1}_{\infty}}E_i\nu^i,\qquad Q_b=\frac{1}{4\pi}\int_{S^{2}_{\infty}}B_i\nu^i.
\end{equation}Note that $Q_b$ is only defined for $n=3$. For $n=4$, there is no total magnetic charge because $B$ is a 2-form and we cannot integrate it over $S^3_{\infty}$. However, for black ring initial data set we have $H_2(M^4)\neq 0$ and there exist a local dipole charge, see \cite{alaee2019existence} for definitions of total electric charge and local dipole charge in the minimal supergravity. For non-smooth initial data set we consider the following class with corners along an axially symmetric hypersurface $\Sigma$ which is a generalization of analogous definition for Riemannian manifolds with corners by Miao \cite[Definition 1]{miao2002positive}.
\begin{definition}
	An axially symmetric initial data set  $(\mathcal{G},\mathcal{K},\mathcal{E},\mathcal{B})$ admitting corners across an axially symmetric hypersurface $\Sigma$ is defined to be  $\mathcal{G}=(g_-,g_+)$, $\mathcal{K}=(k_-,k_+)$, $\mathcal{E}=(E_-,E_+)$, and $\mathcal{B}=(B_-,B_+)$  where $(g_-,k_-,E_-,B_-)$ and $(g_+,k_+,E_+,B_+)$ are initial data sets on $\Omega$ and $M\backslash\bar{\Omega}$ respectively such that the metrics are $C^{2}$ up to the boundary and second fundamental form, electric field, and magnetic $(n-2)$-form are $C^1$ up to the boundary. Moreover, they reduced the same metric, charge potentials, and twist potentials on $\Sigma$.
\end{definition}Note that the second fundamental form, electric field, and magnetic $(n-2)$-form can be discontinuous on $\Sigma$. The continuity of potentials are related to conservation of angular momentum and charges for homologous surfaces to $S^{n-1}_{\infty}$. 
\begin{theorem}\label{thm1.2}
	Let $(M^n,\mathcal{G},\mathcal{K},\mathcal{E},\mathcal{B})$ be a n-dimensional axially symmetric, simply connected, maximal initial data set satisfying constraint equations \eqref{Hcons}, \eqref{momeq}, and \eqref{EMfield} and admitting
	corners across an axially symmetric hypersurface $\Sigma$ with two ends, one designated asymptotically flat and the other either asymptotically flat or asymptotically cylindrical. Suppose that
	$\mu\geq 0$ and $J(\eta_{(l)})=0$, for $l=1,n-2$, in $\Omega$ and $M^n\,\backslash\bar{\Omega}$, and 
	\begin{equation}\label{MCJ}
	H_-(\Sigma,g_-)\geq H_+(\Sigma,g_+)
	\end{equation}where $H_-(\Sigma,g_-)$ and $H_-(\Sigma,g_-)$ represent the mean curvature of $\Sigma$ in $(\bar{\Omega}, g_-)$
	and $(M^n\,\backslash{\Omega}, g_-)$ with respect to unit normal vectors pointing to the designated asymptotically flat region.
	\begin{enumerate}[label=(\alph*)]
		\item If $n=3$, then 
		\begin{equation*}
	m^2\geq \frac{Q^2+\sqrt{Q^4+4\mathcal{J}^2}}{2}.
	\end{equation*}Moreover, equality holds if and only if $(\mathcal{G},\mathcal{K},\mathcal{E},\mathcal{B})$ is isometric to the canonical
	slice of an extreme Kerr-Newman spacetime.
	\item If $n=4$ and $H_2(M^4)=0$, then
	\begin{equation*}
	m\geq \frac{27\pi}{8}\frac{\left(\mathcal{J}_1+\mathcal{J}_2\right)^2}{\left(2m+\sqrt{3}|Q|\right)^2}+\sqrt{3}|Q|
	\end{equation*}Moreover, equality holds if and only if $(\mathcal{G},\mathcal{K},\mathcal{E},\mathcal{B})$ is isometric to the canonical
	slice of an extreme charge-Myers-Perry spacetime.
	\item If $n=4$, $\mathcal{E}=\mathcal{B}=0$, $M^4\cong S^2\times B^2\# \mathbb{R}^4$, $\mathcal{J}_1\geq \mathcal{J}_2$, and the second end is either Kaluza-Klein-asymptotically flat or asymptotically cylindrical, then 
	\begin{equation*}
	m^3\geq \frac{27\pi}{4}|\mathcal{J}_2||\mathcal{J}_1-\mathcal{J}_2|
	\end{equation*}Moreover, if $\mathcal{J}_1> \mathcal{J}_2$, then equality holds if and only if $(\mathcal{G},\mathcal{K})$ arise from the canonical slice of an extreme Pomeransky-Se\'nkov black ring spacetime.
\end{enumerate}
\end{theorem}
For black ring initial data set, the manifold is $M^4\cong S^2\times B^2\# \mathbb{R}^4$ \cite{alaee2014notes,khuri2018plumbing}. In this manifold axially symmetric hypersurfaces $\Sigma$ are diffeomorphic to $3$-sphere and $S^1\times S^2$. Moreover, the second end has topology $S^1\times S^2$ which can geometrically divide to Kaluza-Klein-asymptotically flat or asymptotically cylindrical.

It should be pointed out that the hypotheses used in Theorem \ref{thm1.2} is strong, but it is necessary to address the rigidity cases. In general, we remove simple connectivity and completeness for non-smooth initial data sets in three and four dimensions and prove strict inequalities which are generalization of \cite{bryden2018positive}.  
\begin{cor}\label{cor1}
Consider the initial data set $(M^n,\mathcal{G},\mathcal{K},\mathcal{E},\mathcal{B})$ in Theorem \ref{thm1.2} that have a boundary $\partial M^n$ with non-positive mean curvature, with respect to unit normal vectors pointing to the designated asymptotically flat region, instead of a second end and without simple connectivity assumption. Assume the outermost minimal surface $\Sigma_{\text{min}}$ in $M^n$ has one component and enclosed by corner $\Sigma$.
\begin{enumerate}[label=(\alph*)]
	\item If $n=3$, then the strict inequality in Theorem \ref{thm1.2}-(a) holds.
	\item If $M^4$ is spin, $\pi_1(\Sigma_{\text{min}})=0$, and $H_2(M^4\,\backslash W)=0$, where $\partial W=\partial M^n\cup\Sigma_{\text{min}}$, then the strict inequality in Theorem \ref{thm1.2}-(b) holds.. 
	\item If $M^4$ is spin, $\mathcal{E}=\mathcal{B}=0$, $\pi_1(\Sigma_{\text{min}})=\mathbb{Z}$, $H_2(M^4\,\backslash W)=\mathbb{Z}$ and $\mathcal{J}_1\geq \mathcal{J}_2$, then the strict inequality in Theorem \ref{thm1.2}-(c) holds. 
\end{enumerate} 
\end{cor}

The organization of this paper is as follows. In Section \ref{sec2}, first we construct potentials for our initial data sets. Then we show any axially symmetric maximal initial data sets have related $(t-\phi^i)$ symmetric initial data sets with same mass, angular momentum, and charges. Furthermore, we deform initial data set to construct a $C^2$ metric and $C^1$ second fundamental, electric field and magnetic $(n-2)$-form such that it is a solution of Maxwell equations and has vanishing energy flux in the direction of axial symmetry. In section \ref{sec3}, we construct a conformal transformation of initial data set such that the conformal data has non-negative energy density as well as to be a solution of Maxwell equations and has vanishing energy flux in the direction of axial symmetry. Finally, we prove the main results.
\section{Smooth Deformations of initial data sets across $\Sigma$}\label{sec2}
In this section, we assume $M^n$ is a simply connected, asymptotically flat Riemannian manifold with two ends. Before deforming the initial data smoothly we need to construct potentials that characterize angular momentums and charges of initial data sets and we have two remarks about our setting.
\begin{remark}\label{rem3.1}The Hodge star operator $\star$ is an isomorphism from $p$-form on $(M,g)$ to $(p-1)$-form and defined by $\alpha\wedge\star\beta=\left<\alpha,\beta\right>\star 1=\left<\alpha,\beta\right> dV_g$ for $p$-forms $\alpha$ and $\beta$. In particular, $\star^2\alpha=(-1)^{p(n-p)}\alpha$, $\iota_X\star\alpha=\star\left(\alpha\wedge X\right)$, and $\text{div}_{g}X=(-1)^{n+1}\star d\star X$, where $X$ is a vector field and for dual 1-form we use same notation.
\end{remark} 
\begin{remark}In order to analyze three and four dimensional initial data sets together, we define vectors with components 1 and $n-2$ for 1-forms, vector fields, and functions. In particular, we write the vector $\eta=(\eta_{(1)},\eta_{(n-2)})^T$ for generators of $U(1)^{n-2}$ symmetry of $n$-dimensional initial data sets. Then if $n=3$, $\eta=\eta_{(1)}$ and if $n=4$ we have $\eta=(\eta_{(1)},\eta_{(2)})^T$. 
\end{remark}
\subsection{Potentials}Consider the initial data set $(M^n,\mathcal{G},\mathcal{K},\mathcal{E},\mathcal{B})$ on region ${\Omega}$ and $M^n\,\backslash\bar{\Omega}$. Define the following $1$-form $\Upsilon_{\pm}=(\Upsilon^1_{\pm},\Upsilon^{n-2}_{\pm})^T$ from magnetic $(n-2)$-form $B_{\pm}$ as following
\begin{equation}
\Upsilon_{\pm}=\iota_{\eta}\star_{\pm} B_{\pm}=\left(\iota_{\eta_{(1)}}\star_{\pm} B_{\pm},\iota_{\eta_{(n-2)}}\star_{\pm} B_{\pm}\right)^T,
\end{equation}where  $\iota$
is the interior product and $\star_{\pm}$ is the Hodge star operation with respect to metrics $g_{\pm}$. Then the divergence free property of the magnetic $n-2$-form equation \eqref{EMfield} and Cartan's magic formula shows that the 1-form $\Upsilon$ is closed and by the Poincar\'e lemma, it is exact. Thus there exist potentials
\begin{equation}\label{eqpsi}
d\psi_{\pm}=\left(d\psi^1_{\pm},d\psi^{n-2}_{\pm}\right)^T=\Upsilon_{\pm}.
\end{equation}Since $B_{\pm}$ is invariant under the $U(1)^{n-2}$ symmetry, these potentials are invariant under $U(1)^{n-2}$ symmetry. We define another $1$-form $\mathcal{Q}_{\pm}=\iota_{\eta_{(n-2)}}\iota_{\eta_{(1)}}\star_{\pm} E_{\pm}-\frac{n-3}{\sqrt{3}}\psi_{\pm}^T\mathfrak{J}d\psi_{\pm}$ where 
\begin{equation}
\mathfrak{J}=\left(
\begin{array}{cc}  0 & 1 \\
-1 & 0  \end{array} \right).
\end{equation}Observe that vector $\psi^T\mathfrak{J}$ is orthogonal to vector $\psi$. Then with Cartan's magic formula, Remark \ref{rem3.1}, and Maxwell equations \eqref{EMfield}, we have
\begin{equation}
\begin{split}
d\mathcal{Q}_{\pm}&=d\iota_{\eta_{(n-2)}}\iota_{\eta_{(1)}}\star_{\pm} E_{\pm}-\frac{n-3}{\sqrt{3}}d\left(\psi_{\pm}^T\mathfrak{J}d\psi_{\pm}\right)\\
&=\iota_{\eta_{(n-2)}}\iota_{\eta_{(1)}}d\star_{\pm} E_{\pm}-\frac{n-3}{\sqrt{3}}d\left(\psi_{\pm}^T\mathfrak{J}d\psi_{\pm}\right)\\
&=\iota_{\eta_{(n-2)}}\iota_{\eta_{(1)}}\left(-\frac{n-3}{\sqrt{3}}B_{\pm}\wedge B_{\pm}\right)-\frac{n-3}{\sqrt{3}}d\left(\psi_{\pm}^T\mathfrak{J}d\psi_{\pm}\right)\\
&=-\frac{n-3}{\sqrt{3}}\left(\iota_{\eta_{(n-2)}}\iota_{\eta_{(1)}}\left(B_{\pm}\wedge B_{\pm}\right)+d\left(\psi_{\pm}^T\mathfrak{J}d\psi_{\pm}\right)\right),
\end{split}
\end{equation}where we take $\star_{\pm}$ of Maxwell equation \eqref{EMfield} to get third line. For $n=3$, this expression vanishes. Assuming $n=4$, then using $\star_{\pm} B_{\pm}\wedge \star_{\pm}B_{\pm}=B_{\pm}\wedge B_{\pm}$ for a $2$-form on $4$-manifold and definition of $\psi_{\pm}$ in \eqref{eqpsi}, the expression also vanishes. Hence $\mathcal{Q}_{\pm}$ is closed and there exist a global electric potential $\chi_{\pm}$ so that
\begin{equation}\label{chi}
d\chi_{\pm}=\iota_{\eta_{(n-2)}}\iota_{\eta_{(1)}}\star_{\pm} E_{\pm}-\frac{n-3}{\sqrt{3}}\psi_{\pm}^T\mathfrak{J}d\psi_{\pm}.
\end{equation}Moreover, it immediately follows that this potential is invariant under the $U(1)^{n-2}$ symmetry. Next we define another $1$-form 
\begin{equation}
\Xi_{\pm}=(n-2)\star_{\pm}(k_{\pm}(\eta)\wedge \eta_{(1)}\wedge \eta_{(n-2)})-\psi_{\pm}\left(d\chi_{\pm}+\frac{n-3}{3\sqrt{3}}\psi_{\pm}^T\mathfrak{J}d\psi_{\pm}\right)-\left(n-4\right)\chi_{\pm}d\psi_{\pm}
\end{equation}Then using again Cartan's magic formula, Remark \ref{rem3.1}, $J(\eta_{(l)})=0$, the momentum constraint equation \eqref{momeq}, and \eqref{chi}, we have 
\begin{equation}
\begin{split}
d\Xi_{\pm}&=(n-2)d\star_{\pm}(k_{\pm}(\eta)\wedge \eta_{(1)}\wedge \eta_{(n-2)})-d\psi_{\pm}\wedge\left(d\chi_{\pm}+\frac{n-3}{\sqrt{3}}\psi_{\pm}^T\mathfrak{J}d\psi_{\pm}\right)\\
&\quad-\left(n-4\right)d\chi_{\pm}\wedge d\psi_{\pm}\\
&=(n-2)\iota_{\eta_{(n-2)}}\iota_{\eta_{(1)}}\star_{\pm}\star_{\pm}d\star_{\pm}k_{\pm}(\eta)-d\psi_{\pm}\wedge\left(d\chi_{\pm}+\frac{n-3}{\sqrt{3}}\psi_{\pm}^T\mathfrak{J}d\psi_{\pm}\right)\\
&\quad-\left(n-4\right)d\chi_{\pm}\wedge d\psi_{\pm}\\
&=\frac{-2}{(n-2)}\iota_{\eta_{(n-2)}}\iota_{\eta_{(1)}}\left(\star_{\pm} E_{\pm}\wedge d\psi_{\pm}\right)-d\psi_{\pm}\wedge\left(d\chi_{\pm}+\frac{n-3}{\sqrt{3}}\psi_{\pm}^T\mathfrak{J}d\psi_{\pm}\right)\\
&\quad-\left(n-4\right)d\chi_{\pm}\wedge d\psi_{\pm}\\
&=\frac{-2}{(n-2)}\left(d\chi_{\pm}+\frac{n-3}{\sqrt{3}}\psi_{\pm}^T\mathfrak{J}d\psi_{\pm}\right)\wedge d\psi_{\pm}-d\psi_{\pm}\wedge \left(d\chi_{\pm}+\frac{n-3}{\sqrt{3}}\psi_{\pm}^T\mathfrak{J}d\psi_{\pm}\right)\\
&\quad-\left(n-4\right)d\chi_{\pm}\wedge d\psi_{\pm}=0
\end{split}
\end{equation}Hence, there exist twist potentials
\begin{equation}
d\zeta_{\pm}=\Xi_{\pm}
\end{equation} where $\zeta_{\pm}=(\zeta^1_{\pm},\zeta^{n-2}_{\pm})^T$. We assume the potentials are continuous on $\Sigma$ and define the following functions 
\begin{equation}
\zeta:=\begin{cases}
\zeta_- & \text{on}\quad \Omega\\
\zeta_-=\zeta_+& \text{on}\quad \Sigma\\
\zeta_+ & \text{on}\quad M\backslash\bar{\Omega}
\end{cases},\quad \psi:=\begin{cases}
\psi_- & \text{on}\quad \Omega\\
\psi_-=\psi_+& \text{on}\quad \Sigma\\
\psi_+ & \text{on}\quad M\backslash\bar{\Omega}
\end{cases},
\end{equation} and
\begin{equation}
\chi:=\begin{cases}
\chi_- & \text{on}\quad \Omega\\
\chi_-=\chi_+& \text{on}\quad \Sigma\\
\chi_+ & \text{on}\quad M\backslash\bar{\Omega}
\end{cases}.
\end{equation}In particular, a computation using definition of angular momentum \eqref{AM} and charges \eqref{charges} shows that, see \cite{alaee2017relating,dain2013lower} for details,
\begin{equation}\label{AmCrelation}
\mathcal{J}_l=\frac{\pi^{n-3}}{4}\left(\zeta_l\vert_{\Gamma_-}-\zeta_l\vert_{\Gamma_+}\right),\quad Q_{e}=\frac{\pi^{n-3}}{2^{n-2}}\left(\chi\vert_{\Gamma_-}-\chi\vert_{\Gamma_+}\right),\quad Q_{b}=\frac{1}{2}\left(\psi\vert_{\Gamma_-}-\psi\vert_{\Gamma_+}\right),
\end{equation} where $\Gamma_+$ and $\Gamma_-$ are two asymptotic parts of axis of rotation $\Gamma$. Since the value of potentials at axis is the same in $\bar{\Omega}$ and $M^n\,\backslash{\Omega}$, charges and angular momentum are conserve quantities for any surface homologous to coordinate sphere $S^{n-1}$ at infinity of $M^n\,\backslash {\Omega}$. 
\subsection{$(t-\phi^i)$-symmetric data}In this section, we show that every axially symmetric maximal data set has related $(t-\phi^i)$-symmetric data with same mass, angular momentum, and charges. Consider axially symmetric, $C^2$ initial data set $(M^n,g,k,E,B)$ with two ends. Then there exist a global Brill coordinate $(\rho,z,\phi^l)$ such that $\eta_{(l)}=\frac{\partial}{\partial\phi^l}$. This has been proved in three dimensions \cite{chrusciel2008mass} and it is conjectured to be true in four dimensions \cite{alaee2015proof}. In particular, the metric of axially symmetric initial data sets takes the form
\begin{equation}\label{axialmetric}
{g}=e^{-2U+2\alpha}\left(d\rho^2+dz^2\right)+e^{-2U}{\lambda}_{ij}\left(d\phi^i+A_{a}^idy^a\right)\left(d\phi^j+A_{a}^jdy^a\right),
\end{equation}for some functions $U$, $\alpha$, $A^i_{a}$, and a symmetric positive definite matrix ${\lambda}=[{\lambda}_{ij}]$ with $\det{\lambda}=\rho^2$, all independent of $(\phi^{1},\phi^{n-2})$ with asymptotic fall-off in \cite{alaee2015proof,khuri2016positive}. Note that $\alpha$ in equation \eqref{axialmetric} is equal to ${\alpha}-\log(2\sqrt{\rho^2+z^2})$ in \cite{alaee2015proof}. Moreover, the coordinates should take values in the following ranges $\rho\in[0,\infty)$, $z\in \mathbb{R}$, and $\phi^i\in [0,2\pi]$, for $i=1,n-2$. The transformation to spherical coordinate in three dimensions is $\rho=r\sin\theta$ and $z=r\cos\theta$ and in four dimensions is $\rho=\frac{r^2}{2}\sin 2\theta$ and $z=\frac{r^2}{2}\cos 2\theta$. Consider the global frame
\begin{equation}\label{frame1}
e_1= e^{U-\alpha}\left(\partial_\rho - A^{i}_{\rho} \partial_{\phi^i}\right), \text{ }\text{ }\text{ }e_2 = e^{U-\alpha}\left(\partial_z - A^{i}_{z} \partial_{\phi^i}\right),\text{ }\text{ }\text{ }
e_{i+2} = e^{U} \partial_{\phi^i},\text{ }\text{ }
\end{equation}
for $i=1,n-2,$ with dual co-frame
\begin{equation}\label{coframe1}
\theta^1= e^{-U+\alpha}d\rho,\text{ } \text{ } \text{ } \text{ } \text{ } \text{ }
\theta^2= e^{-U+\alpha}dz,\text{ }\text{ }\text{ }\text{ }\text{ }\text{ }
\theta^{i+2}=e^{-U}\left( d\phi^i+A^i_a d y^a\right),\text{ }\text{ }\text{ }\text{ }
\end{equation}for $i=1,n-2.$ For arbitrary axially symmetric function $\omega=(\omega^1,\omega^{n-2})^T$, we define a 1-form
\begin{equation}
{\mathcal{P}}^{\omega}=\frac{1}{\det \Lambda}\star\left(d\omega\wedge\eta_{(1)}\wedge\eta_{(n-2)}\right)
\end{equation}
where ${\mathcal{P}}^{\omega}=\left({\mathcal{P}}^{\omega}_1,{\mathcal{P}}^{\omega}_{n-2}\right)^T$, $\Lambda=[\Lambda_{ij}]=e^{-2U}[\lambda_{ij}]$ and $\star$ is Hodge star operation with respect to $g$. Then we have the following decomposition for  electric 1-form and second fundamental form
\begin{equation}\label{E}
E=\mathcal{P}^{\chi}+\frac{n-3}{\sqrt{3}}\psi^T\mathfrak{J}\mathcal{P}^{\psi}+\hat{E},
\end{equation}
\begin{equation}\label{extrin}
k(e_i,e_j)=\frac{1}{(n-2)}\left(\eta^T(e_i)\Lambda^{-1}{\mathcal{P}}(e_j)+\eta^T(e_j)\Lambda^{-1}{\mathcal{P}}(e_i)\right)+{\pi}(e_i,e_{j+2}),
\end{equation}
where $i,j=1,\cdots,4$, $\pi(e_k,e_{l+2})=0$, $\hat{E}(e_k)=0$ for $k,l=1,2$ and
\begin{equation}
\mathcal{P}=\mathcal{P}^{\zeta}+\psi\left(\mathcal{P}^{\chi}+\frac{n-3}{3\sqrt{3}}\psi^T\mathfrak{J}\mathcal{P}^{\psi}\right)+\left(n-4\right)\chi\mathcal{P}^{\psi}
\end{equation}For magnetic $(n-2)$-form we have different decompositions related to the dimension. If $n=3$, then we have 
\begin{equation}\label{B1}
B=\mathcal{P}^{\psi}+\hat{B},
\end{equation} where $\hat{B}(e_k)=0$ for $k=1,2$. If $n=4$, the Killing symmetry implies $\star B(\eta_{(i)},\eta_{(j)})=0$ which leads to $B(e_1,e_2)=0$, and 
\begin{equation}\label{B2}
B=B(e_1,e_{k+2})\theta^1\wedge\theta^{k+2}+B(e_2,e_{k+2})\theta^2\wedge\theta^{k+2}+\hat{B}(e_3,e_{4})\theta^3\wedge\theta^{4},
\end{equation}such that \begin{equation}\label{Bbase}
B(e_1,e_{i+2})(t)=-e^{2U-\alpha}{\epsilon}^j{}_i\partial_{z}\psi^j(t),\qquad B(e_2,e_{i+2})(t)=e^{2U-\alpha}{\epsilon}^j{}_i\partial_{\rho}\psi^j(t),
\end{equation} where ${\epsilon_{\delta}}_{ij}=\sqrt{\det\lambda}\varepsilon_{ij}$ is volume form associated to $\lambda$ and $\varepsilon_{ij}=0,\pm 1$.

In general, the 2-dimensional distribution $\mathcal{D}^2$ orthogonal to $\eta_{(l)}$ is not integrable for axially symmetric initial data sets. However, the Killing vector fields $\eta_{(l)}$ have 0-dimensional and 1-dimensional fix points and assuming $Ric_g(\eta,\partial_{y^a})=0$, for $a=1,2$ and $(y^1,y^2)=(\rho,z)$,  the result of Wald \cite[Theorem 7.1.1]{wald2010general} implies $\mathcal{D}^2$ is integrable. In particular, if $\mathcal{D}^2$ is integrable, we have the following divergence-free one-form, see \cite[Section 2.2]{alaee2014small} for 4-dimensional case and \cite[Appendix C]{wald2010general} for 3-dimensional case. 
\begin{lemma}\label{Lemma4.3} Let $(M,\bar{g})$ be a $n$-dimensional Riemannian manifold with $U(1)^{n-2}$ isometry subgroup for $n=3,4$. If $2$-dimensional distribution $\mathcal{D}^2$ orthogonal to $\eta_{(l)}$ is integrable, then the following 1-form is divergence free
	\begin{equation}
	{\mathcal{P}}^{\omega}=\frac{1}{\det \Lambda}\star_{\bar{g}}\left(d\omega\wedge\eta_{(1)}\wedge\eta_{(n-2)}\right),
	\end{equation}where ${\mathcal{P}}^{\omega}=({\mathcal{P}}^{\omega_1},{\mathcal{P}}^{\omega_{n-2}})^T$, $\Lambda=[\Lambda_{ij}]=[\bar{g}(\eta_{(i)},\eta_{(j)})]$ is Gram matrix of the Killing fields, and $\omega=(\omega_1,\omega_{n-2})^T$ is axially symmetric function.
\end{lemma} 
\begin{proof}Since $2$-dimensional distribution orthogonal to $\eta_{(l)}$ is integrable, by Frobenius theorem we have $\nabla^i\eta^j=\frac{1}{2}\left(\nabla^i\Lambda\Lambda^{-1}\eta^j-\nabla^j\Lambda\Lambda^{-1}\eta^i\right)$, where $\nabla$ is covariant derivative with respect to $\bar{g}$ and $\eta=(\eta_1,\eta_{n-2})^T$. A straightforward computation then shows that
	\begin{equation}
	\epsilon_{ijkl}\nabla^i\eta^j_{(1)}\eta^k_{(n-2)}+\epsilon_{ijkl}\eta^j_{(1)}\nabla^i\eta^k_{(n-2)}=\nabla^i\log\det\Lambda \epsilon_{ijkl}\eta_{(1)}^j\eta_{(n-2)}^k,
	\end{equation}where $\epsilon_{ijkl}$ is volume form with respect to $\bar{g}$. It follows that
	\begin{equation}
	\begin{split}
	\text{div}_{\bar{g}}{\mathcal{P}}^{\omega}&=-\frac{1}{\det\Lambda}\nabla^i\log\det\Lambda \epsilon_{ijkl}\eta_{(1)}^j\eta_{(n-2)}^k\nabla^l\omega+\frac{1}{\det\Lambda}\epsilon_{ijkl}\nabla^i\eta^j_{(1)}\eta_{(n-2)}^k\nabla^l\omega\\
	&\quad+\frac{1}{\det\Lambda}\epsilon_{ijkl}\eta^j_{(1)}\nabla^i\eta_{(n-2)}^k\nabla^l\omega
	+\frac{1}{\det\Lambda}\epsilon_{ijkl}\eta_{(1)}^j\eta_{(n-2)}^k\nabla^i\nabla^l\omega= 0.
	\end{split}
	\end{equation}The last term vanishes because $\nabla^i\nabla^l\omega=\nabla^l\nabla^i\omega$ and $\epsilon_{ijkl}$ is antisymmetric in $i$ and $l$.
\end{proof}The class of initial data which has the above condition is called $(t-\phi^i)$-symmetric. This class was defined in three dimensions by Gibbons \cite{gibbons1972time} and was generalized to higher dimensions in \cite{figueras2011black}. We extend the definition to initial data sets with electromagnetic field.

\begin{definition}An axially symmetric initial data set $(M^n,\bar{g},\bar{k},\bar{E},\bar{B})$ is $(t-\phi^i)$-symmetric if $\phi^i\to -\phi^i$, we have $\bar{g}\to \bar{g}$, $\bar{k}\to -\bar{k}$, $\bar{E}\to \bar{E}$, and $\bar{B}\to (-1)^{n-1}\bar{B}$.
\end{definition} Then we have the following proposition.
\begin{proposition}\label{proptphi}
Suppose $(M^n,g,k,E,B)$ is an axially symmetric, maximal initial data set with $\mu\geq 0$ and $J(\eta_{(l)})=0$. Then there exists a $(t-\phi^i)$-symmetric initial data set $(M^n,\bar{g},\bar{k},\bar{E},\bar{B})$ with same mass, angular momentum, and charges. Moreover, $\bar{\mu}\geq \mu$ and $\bar{J}(\eta_{(l)})=0$.
\end{proposition}
\begin{proof}Consider the axially symmetric, maximal initial data set $(M^n,g,k,E,B)$. Then related $(t-\phi^i)$ initial data set  $(M^n,\bar{g},\bar{k},\bar{E},\bar{B})$ is obtained as follows. The metric $\bar{g}$ obtained from $g$ by setting $A^i_{a}=0$ in equation \eqref{axialmetric}. Then we have new global frame $\{e_i\}$ equal to the frame \eqref{frame1} without $A^i_a$ terms. Moreover, setting $\pi=\hat{E}=\hat{B}=0$ in equations \eqref{extrin}, \eqref{E}, \eqref{B1}, and \eqref{B2}, we get $\bar{k}$, $\bar{E}$ and $\bar{B}$, respectively. We define the energy density and energy flux as following 
	\begin{equation}\label{bmu}
	16\pi\bar{\mu}=16\pi\mu+|\pi|^2_{g}+|\hat{E}|_{g}+|\hat{B}|_{g}+\frac{1}{4}e^{-2\alpha}\lambda_{ij}\left(\partial_{\rho}A^i_{z}-\partial_{z}A^i_{\rho}\right)\left(\partial_{\rho}A^j_{z}-\partial_{z}A^j_{\rho}\right),
	\end{equation}
	\begin{equation}
	8\pi \bar{J}=\text{div}_{\bar{g}}\bar{k}+\frac{2}{(n-2)^2}\star_{\bar{g}}\left(\bar{B}\wedge \bar{E}\right).
	\end{equation}It follows that $(M^n,\bar{g},\bar{k},\bar{E},\bar{B})$ is a solution of Hamiltonian constraint equation with $\bar{\mu}\geq \mu$. Next we need to show that $\bar{J}(\eta_{(l)})=0$. Since 2-dimensional distribution orthogonal to $\eta_{(l)}$ is integrable, the 1-forms $\mathcal{P}^{\chi}$, $\mathcal{P}^{\psi}$, and $\mathcal{P}^{\zeta}$ defined in Lemma \ref{Lemma4.3} with respect to metric $\bar{g}$ are divergence-free. Using axially symmetry condition, the divergence of $\mathcal{P}$ is
\begin{equation}\label{p1}
\begin{split}
\text{div}_{\bar{g}}\mathcal{P}&=\sum_{i=1}^{2}e_i\left(\psi\right)\left(\mathcal{P}^{\chi}(e_i)+\frac{n-3}{3\sqrt{3}}\psi^T\mathfrak{J}\mathcal{P}^{\psi}(e_i)\right)+\frac{n-3}{3\sqrt{3}}\psi\sum_{i=1}^{2}e_i\left(\psi^T\right)\mathfrak{J}\mathcal{P}^{\psi}(e_i)\\
&\quad+\left(n-4\right)\sum_{i=1}^{2}e_i\left(\chi\right)\mathcal{P}^{\psi}(e_i)\\
&=\sum_{i=1}^{2}e_i\left(\psi\right)\left(\mathcal{P}^{\chi}(e_i)+\frac{n-3}{\sqrt{3}}\psi^T\mathfrak{J}\mathcal{P}^{\psi}(e_i)\right)+\left(n-4\right)\sum_{i=1}^{2}e_i\left(\chi\right)\mathcal{P}^{\psi}(e_i)\\
&=\sum_{i=1}^{2}e_i\left(\psi\right)\left((5-n)\mathcal{P}^{\chi}(e_i)+\frac{n-3}{\sqrt{3}}\psi^T\mathfrak{J}\mathcal{P}^{\psi}(e_i)\right)\\
&=\frac{2}{(n-2)}\sum_{i=1}^{2}e_i\left(\psi\right)\left(\mathcal{P}^{\chi}(e_i)+\frac{n-3}{\sqrt{3}}\psi^T\mathfrak{J}\mathcal{P}^{\psi}(e_i)\right),
\end{split}
\end{equation}where we used $e_i\left(\chi\right)\mathcal{P}^{\psi}(e_i)=-e_i\left(\psi\right)\mathcal{P}^{\chi}(e_i)$. The last equality is true if $n=3,4$. On the the hand, we have
\begin{equation}\label{p2}
\begin{split}
i_{\eta_{(l)}}\star_{\bar{g}}\left(\bar{B}\wedge \bar{E}\right)&=-\left<d\psi^l,\bar{E}\right>_{\bar{g}}\\
&=-\sum_{i=1}^{2}e_i(\psi^l)\left(\mathcal{P}^{\chi}(e_i)+\frac{n-3}{\sqrt{3}}\psi^T\mathfrak{J}\mathcal{P}^{\psi}(e_i)\right).
\end{split}
\end{equation}Combining \eqref{p1} and \eqref{p2} shows that
\begin{equation}\label{divp}
\text{div}_{\bar{g}}\mathcal{P}=-\frac{2}{(n-2)}i_{\eta_{l}}\star_{\bar{g}}\left(B\wedge E\right).
\end{equation}
Observe that $\bar{k}$ has the following decomposition 
\begin{equation}\label{ecd}
\bar{k}(e_i,e_j)=\frac{1}{(n-2)}\left(\eta^T(e_i)\Lambda^{-1}{\mathcal{P}}(e_j)+\eta^T(e_j)\Lambda^{-1}{\mathcal{P}}(e_i)\right),
\end{equation}where $i,j=1,\cdots,4$ and clearly $\text{tr}_{\bar{g}}\bar{k}=0$. 
Then 
\begin{equation}\label{divk}
\begin{split}
\text{div}_{\bar{g}}\left(\bar{k}(\eta_{(l)})\right)=\left(\text{div}_{\bar{g}}\bar{k}\right)(\eta_{(l)})+\frac{1}{2}\left<k,\mathcal{L}_{\eta_{(l)}}\bar{g}\right>=\left(\text{div}_{\bar{g}}\bar{k}\right)(\eta_{(l)}).
\end{split}
\end{equation} Since $\bar{g}\left(\mathcal{P},\eta_{(l)}\right)=0$, we have
\begin{equation}\label{keta}
\bar{k}(\eta_{(l)},\cdot)=\frac{1}{(n-2)}{\mathcal{P}}.
\end{equation}Now from equations \eqref{divp}, \eqref{divk}, and \eqref{keta}, we have 
\begin{equation}
\bar{J}(\eta_{(l)})=\left(\text{div}_{\bar{g}}\bar{k}\right)(\eta_{(l)})+\frac{2}{(n-2)^2}i_{\eta_{(l)}}\star_{\bar{g}}\left(\bar{B}\wedge \bar{E}\right)=0,
\end{equation}for $l=1,n-2$. Next we need to show Maxwell equations for initial data set  $(M^n,\bar{g},\bar{k},\bar{E},\bar{B})$. 
Let $\mathcal{P}^{\chi}$ be the 1-form given in Lemma \ref{Lemma4.3} with respect to metric $\bar{g}$. Then using $\hat{E}=0$ and \eqref{E}, the divergence of $\bar{E}$ is
\begin{equation}\label{e1}
\text{div}_{\bar{g}}\bar{E}=\text{div}_{\bar{g}}\mathcal{P}^{\chi}+\frac{n-3}{\sqrt{3}}\text{div}_{\bar{g}}\left(\psi^T\mathfrak{J}\mathcal{P}^{\psi}\right).
\end{equation}Using definition of $\bar{B}$, equation \eqref{Bbase}, and Lemma \ref{Lemma4.3}, we have
\begin{equation}
\begin{split}
\text{div}_{\bar{g}}(\psi^T\mathfrak{J}\mathcal{P}^{\psi})&=\sum_{i=1}^{2}e_i\left(\psi^T\right)\mathfrak{J}\mathcal{P}^{\psi}(e_i)=\sum_{i=1}^{2}\left(e_i\left(\psi^1\right)\mathcal{P}^{\psi^2}(e_i)-e_i\left(\psi^2\right)\mathcal{P}^{\psi^1}(e_i)\right)\\
&=\star_{\bar{g}}\left(\bar{B}\wedge \bar{B}\right)
\end{split}
\end{equation}In light of this, and the fact that the first term of \eqref{e1} vanishes by Lemma \eqref{Lemma4.3}, it follows that 
\begin{equation}
\text{div}_{\bar{g}}\bar{E}=\frac{n-3}{\sqrt{3}}\star_{\bar{g}}\left(\bar{B}\wedge \bar{B}\right).
\end{equation} Note that when $n=3$, similarly we can show $\text{div}_{\bar{g}}\bar{B}=0$, but for $n=4$, we have magnetic $2$-form in global frame $\{e_i\}$. Then, we have
\begin{equation}
\begin{split}
\star_{\bar{g}} \bar{B}&={\epsilon_{\bar{g}}}_{2(j+2)}{}^{1(i+2)}\bar{B}(e_1,e_{i+2})\theta^2\wedge\theta^{j+2}+{\epsilon_{\bar{g}}}_{1(j+2)}{}^{2(i+2)}\bar{B}(e_2,e_{i+2})\theta^1\wedge\theta^{j+2}\\
&={\epsilon_{\bar{g}}}_{j}{}^{i}\bar{B}(e_1,e_{i+2})\theta^2\wedge\theta^{j+2}-{\epsilon_{\bar{g}}}_{j}{}^{i}\bar{B}(e_2,e_{i+2})\theta^1\wedge\theta^{j+2}\\
&=-{\epsilon_{\delta}}_{j}{}^{i}{\epsilon_{\delta}}^{j}{}_{i}e_2(\psi^j)\theta^2\wedge\theta^{j+2}-{\epsilon_{\delta}}_{j}{}^{i}{\epsilon_{\delta}}^{j}{}_{i}e_1(\psi^j)\theta^1\wedge\theta^{j+2}\\
&=-e_2(\psi^j)\theta^2\wedge\theta^{j+2}-e_1(\psi^j)\theta^1\wedge\theta^{j+2}.
\end{split}
\end{equation}Since $e_1(e_2(\psi))=e_2(e_1(\psi))$, we have $d\star_{\bar{g}}\bar{B}=0$ which is equivalent to $\text{div}_{\bar{g}}\bar{B}=0$.
 
Since charge and twist potentials of axially symmetric and related $(t-\phi^i)$ symmetric data sets are the same, in light of equation \eqref{AmCrelation}, angular momentum and charges do not change. Moreover, a computation in \cite{alaee2015remarks,chrusciel2008mass} shows that the ADM mass is only depends on $U$ and $\alpha$. Since both initial data sets have same $U$ and $\alpha$, they have the same mass. This complete the proof.
\end{proof}
\begin{remark}\label{remarkmean}Assume $(M^n,\mathcal{G},\mathcal{K},\mathcal{E},\mathcal{B})$ is axially symmetric maximal initial data with corners along an axially symmetric hypersurface $\Sigma$ as in Theorem \ref{thm1.2}. In each region $(\bar{\Omega},g_-)$ and $(M^n\,\backslash\Omega,g_+)$ with related Brill coordinate $(\rho,z,\phi^i)$, if the axially symmetric surface is constant $r=r_0$ surface $\Sigma$, then the mean curvatures are
\begin{equation}\label{mean}
H_{\pm}(\Sigma,g_{\pm})=\frac{2(n-2)/r+\partial_r\left(\alpha_{\pm}-2(n-2)U_{\pm}\right)}{\sqrt{e^{-2U_{\pm}+2\alpha_{\pm}}+e^{-2U_{\pm}}\lambda_{ij}A^i_{\pm}A^j_{\pm}}}\bigg\vert_{r=r_0},
\end{equation}where $A_{\pm}^i=r^{n-3}\left({A^i_{\pm\rho}}\sin(n-2)\theta +{A^i_{\pm z}}\cos(n-2)\theta\right)$.
\end{remark}
\begin{remark}\label{rem2}The $(t-\phi^i)$ symmetric initial data set $(M^n,\bar{\mathcal{G}},\bar{\mathcal{K}},\bar{\mathcal{E}},\bar{\mathcal{B}})$ with corners across an axially symmetric hypersurface $\Sigma$ involves functions $\Psi=(U,\lambda_{ij},\zeta^1,\zeta^{n-2},\chi,\psi^1,\psi^{n-2})$ and $\alpha$. Moreover, since it encodes all geometrical and physical properties of related axially symmetric maximal initial data set $(M^n,\mathcal{G},\mathcal{K},\mathcal{E},\mathcal{B})$, in Section \ref{sdeform}, we study smooth deformations of this data set.
\end{remark}
\subsection{Smooth deformation}\label{sdeform} Consider the $(t-\phi^i)$ symmetric initial data set $(M^n,\bar{\mathcal{G}},\bar{\mathcal{K}},\bar{\mathcal{E}},\bar{\mathcal{B}})$ with corners along an axially symmetric hypersurface $\Sigma$. For this initial data set, the 2-dimensional distribution $\mathcal{D}^2$ orthogonal to $\eta_{(l)}$ is integrable globally. Then we consider the orbit space $O=M^n/U(1)^{n-2}$ and we modify the differential structure on $O$ such that the metric on neighborhood of $\Sigma/U(1)^{n-2}$ has Gaussian normal form. The orbit space has two region $O_{-}=\bar{\Omega}/U(1)^{n-2}$ and $O_{+}=\left(M \backslash{\Omega}\right)/U(1)^{n-2}$. The metric of orbit space is 
\begin{equation}
\pi^{\ast}\left(q_{\pm}\right)=\bar{g}_{\pm}-\Lambda^{kl}_{\pm}\eta_{(k)}\eta_{(l)},
\end{equation}where $\pi:M\to O_{\pm}$ is projection onto orbit space. The orbit space is a 2-dimensional smooth manifold with boundary and corners \cite[Proposition 1]{hollands2008uniqueness} which is diffeomorphic to upper-half plane by the Riemann mapping theorem and $\Sigma/U(1)^{n-2}$ is a semi-circle in this upper-half plane. Given $\epsilon>0$, let $V^{2\epsilon}_-$ and $V^{2\epsilon}_+$ be neighborhood of $\Sigma/U(1)^{n-2}$ in $({O}_-,q_-)$ and $(O_+,q_+)$, respectively. Moreover, Let 
\begin{equation}
\Phi_-:(-2\epsilon,0]\times \Sigma/U(1)^{n-2}\to V^{2\epsilon}_-,\qquad \Phi_+:[0,2\epsilon)\times \Sigma/U(1)^{n-2}\to V^{2\epsilon}_+,
\end{equation} be diffeomorphisms such that the pullback metrics are 
\begin{equation}
\Phi_-^*(q_-)=dt^2+q_-{}_{\theta}(t,\theta)d\theta^2,\qquad \Phi_+^*(q_+)=dt^2+q_+{}_{\theta}(t,\theta)d\theta^2.
\end{equation}Note that level sets of the distance function $d_{\pm}:O_{\pm}\to \mathbb{R}$ to $\Sigma$ provide such diffeomorphisms.
Identifying $V=V^{2\epsilon}_-\cup V^{2\epsilon}_+$ with $ (-2\epsilon,2\epsilon)\times \Sigma/U(1)^{n-2}$, we define a differential structure with open covering $\{O_-, O_+,V\}$ on $O$ and denote the new smooth manifold $\tilde{O}$. In particular, metric on $V=(-2\epsilon,2\epsilon)\times \Sigma/U(1)^{n-2}$ takes the form
\begin{equation}
q=dt^2+q_{\theta}(t,\theta)d\theta^2,
\end{equation} where $q=\Phi_-^*(q_-)$ for $t\leq 0$ and $q=\Phi_+^*(q_+)$ for $t\geq 0$. Moreover, the metric $\mathcal{G}$ on related region $(-2\epsilon,2\epsilon)\times \Sigma$ on $\tilde{M}$ is a continuous metric $\bar{g}$ as following
\begin{equation}
\bar{g}=dt^2+\gamma(t,\theta)=dt^2+q_{\theta}(t,\theta)d\theta^2+\Lambda_{ij}(t,\theta)d\phi^id\phi^j,
\end{equation}where $\eta_{(l)}=\frac{\partial}{\partial\phi^l}$. Then $\gamma(t)=g_{ij}(t,\theta)dx^idx^j$, for $(x^1,x^2,x^{n-1})=(\theta,\phi^1,\phi^{n-2})$, is a path of metrics on $\Sigma$. Suppose $\mathcal{S}^i(\Sigma)$ is Banach space of $C^i$ symmetric $(0,2)$ tensors on $\Sigma$ and $\mathcal{M}^i(\Sigma)$ is open and convex  subset of $\mathcal{S}^i(\Sigma)$ consisting of $C^i$ metrics. Then we have the following path of metrics
\begin{equation}
\gamma: (-2\epsilon,2\epsilon)\rightarrow \mathcal{M}^2(\Sigma)\hookrightarrow\mathcal{M}^1(\Sigma)\hookrightarrow\mathcal{M}^0(\Sigma).
\end{equation} By assumption, $\gamma$ is a continuous path in $\mathcal{M}^2(\Sigma)$ and piecewise $C^1$ path in $\mathcal{M}^1(\Sigma)$. Given $0<\delta\ll \epsilon$, we define the deforming path for $s\in(-\epsilon,\epsilon)$
\begin{equation}
\gamma_{\delta}(s)=\int_{\mathbb{R}}\gamma(s-\sigma_{\delta}(s)t)\phi(t) dt=\begin{cases}
\int_{\mathbb{R}}\gamma(t)\left(\frac{1}{\sigma_{\delta}(s)}\phi(\frac{s-t}{\sigma_{\delta}(s)})\right) dt & \sigma_{\delta}(s)>0\\
\gamma(s) & \sigma_{\delta}(s)=0
\end{cases},
\end{equation}where $\phi(s)\in C^{\infty}_{c}([-1,1])$ is standard mollifier on $\mathbb{R}$ with $0\leq \phi(s)\leq 1$ and $\int_{-1}^1\phi ds=1$. Moreover,  $\sigma_{\delta}(t)=\delta^2\sigma(t/\delta)$ for the cut-off function $\sigma(t)\in C^{\infty}_{c}([-\frac{1}{2},\frac{1}{2}])$ with definition
\begin{equation}
\sigma(t)=\begin{cases}
\sigma(t)= \frac{1}{100} & |t|\leq \frac{1}{4}\\
0<\sigma(t)\leq \frac{1}{100} & \frac{1}{4}<|t|<\frac{1}{2}
\end{cases}.
\end{equation}Then $\gamma_{\delta}$ has the following properties \cite{miao2002positive}.
\begin{lemma}\label{lem1}The deform metric $\gamma_{\delta}$ has the following properties 
	\begin{enumerate}[label=(\alph*)]
		\item $\gamma_{\delta}$ is $C^2$ path in $\mathcal{M}^0(\Sigma)$ and a $C^1$ path in $\mathcal{M}^1(\Sigma)$.
		\item $\gamma_{\delta}$ is $C^0$ path in $\mathcal{M}^2(\Sigma)$ which is uniformly close to $\gamma$ and agrees with $\gamma$ outside $(-\frac{\delta}{2},\frac{\delta}{2})$.
		\item $\norm{\gamma_{\delta}(s)-\gamma(s)}_{\mathcal{M}^0(\Sigma)}\leq L\delta^2$ for $s\in(-\epsilon,\epsilon)$.
	\end{enumerate}
\end{lemma} Now we define the deformation metric
\begin{equation}
\bar{g}_{\delta}=\begin{cases}
dt^2+\gamma_{\delta}(t) & (t,x)\in(-\epsilon,\epsilon)\times\Sigma\\
\bar{g} & (t,x)\notin(-\epsilon,\epsilon)\times\Sigma
\end{cases}.
\end{equation}Before deforming $(\bar{\mathcal{K}},\bar{\mathcal{E}},\bar{\mathcal{B}})$ smoothly on $(-\epsilon,\epsilon)\times \Sigma$, we define the following frame for metric $\bar{g}_{\delta}$  on $(-2\epsilon,2\epsilon)\times \Sigma$
\begin{equation}
e_1=\partial_t,\qquad e_2=\frac{\partial_{\theta}}{\sqrt{\bar{g}_{\delta}(\partial_{\theta},\partial_{\theta})}}\qquad e_{i+2}=\partial_{\phi^i},\qquad \text{for $i=1,n-2$},
\end{equation}with dual frame $\{\theta^i\}$ such that
\begin{equation}\label{defm}
\bar{g}_{\delta}=(\theta^1)^2+(\theta^2)^2+{\Lambda_{\delta}}_{ij}\theta^{i+2}\theta^{j+2}.
\end{equation}Suppose $\omega$ represents the potentials in the initial data set, that is $\omega=\zeta,\chi,$ and $\psi$, then we define 
\begin{equation}\label{deformpo}
\omega_{\delta}(t,x)=\int_{\mathbb{R}}\omega(t-\sigma_{\delta}(t)s)\phi(s) ds.
\end{equation}This implies  $\omega_{\delta}=\omega$ for $(t,x)\notin(-\epsilon,\epsilon)\times\Sigma$. Moreover, since $\omega$ approaches constant on axis $\Gamma$ and $\int_{\mathbb{R}}\phi(t)dt=1$, we have $\omega_{\delta}(t)|_{\Gamma}=\omega(t)|_{\Gamma}$. We define deformations for each components of extrinsic curvature, electric and magnetic $(n-2)$-form in frame $\{e_i\}$. Then using equations \eqref{E}, \eqref{extrin}, \eqref{B1}, \eqref{B2}, \eqref{defm}, and \eqref{deformpo} we obtain the following expressions
\begin{equation}
\begin{split}
&(n-2)\bar{k}_{\delta}(e_1,e_{i+2})(t)\\
&=-\frac{1}{\det\Lambda_{\delta}\sqrt{\bar{g}_{\delta}(\partial_{\theta},\partial_{\theta})}}\left(\partial_{\theta}\zeta^i_{\delta}+\psi_{\delta}^i\left(\partial_{\theta}\chi_{\delta}+\frac{n-3}{3\sqrt{3}}\psi_{\delta}^T\mathfrak{J}d\psi_{\delta}\right)+\left(n-4\right)\chi_{\delta}\partial_{\theta}\psi^i_{\delta}\right)(t),
\end{split}
\end{equation}
\begin{equation}
\begin{split}
&(n-2)\bar{k}_{\delta}(e_2,e_{i+2})(t)\\
&=\frac{1}{\det\Lambda_{\delta}}\left(\partial_{t}\zeta^i_{\delta}+\psi_{\delta}^i\left(\partial_{t}\chi_{\delta}+\frac{n-3}{3\sqrt{3}}\psi_{\delta}^T\mathfrak{J}\partial_t\psi_{\delta}\right)+\left(n-4\right)\chi_{\delta}\partial_{t}\psi^i_{\delta}\right)(t),
\end{split}
\end{equation}for $i=1,2$ and
\begin{equation}
\bar{E}_{\delta}(e_1)(t)=\frac{1}{\sqrt{\bar{g}_{\delta}(\partial_{\theta},\partial_{\theta})\det\Lambda_{\delta}}}\left(\partial_{\theta}\chi_{\delta}+\frac{n-3}{\sqrt{3}}\psi_{\delta}^T\mathfrak{J}\partial_{\theta}\psi_{\delta}\right)(t),
\end{equation}
\begin{equation}
\bar{E}_{\delta}(e_2)(t)=-\frac{1}{\sqrt{\det\Lambda_{\delta}}}\left(\partial_{t}\chi_{\delta}+\frac{n-3}{\sqrt{3}}\psi_{\delta}^T\mathfrak{J}\partial_{t}\psi_{\delta}\right)(t).
\end{equation}For magnetic $(n-2)$-form we have different decomposition in each dimensions. If $n=3$, then we have 
\begin{equation}
\bar{B}_{\delta}(e_1)(t)=\frac{1}{\sqrt{\bar{g}_{\delta}(\partial_{\theta},\partial_{\theta})\det\Lambda_{\delta}}}\partial_{\theta}\psi_{\delta}(t),\qquad \bar{B}_{\delta}(e_2)(t)=-\frac{1}{\sqrt{\det\Lambda_{\delta}}}\partial_{t}\psi_{\delta}(s)(t)
\end{equation} For $n=4$, we have 
\begin{equation}\label{Bbas}
\bar{B}_{\delta}(e_1,e_{i+2})(t)=-\frac{1}{\sqrt{\bar{g}_{\delta}(\partial_{\theta},\partial_{\theta})}}{\epsilon_{\delta}}^j{}_i\partial_{\theta}\psi_{\delta}^j(t),\qquad \bar{B}_{\delta}(e_2,e_{i+2})(t)={\epsilon_{\delta}}^j{}_i\partial_{t}\psi_{\delta}^j(t)
\end{equation} where ${\epsilon_{\delta}}_{ij}=\sqrt{\det\Lambda_{\delta}}\varepsilon_{ij}$ is volume form associated to $\Lambda_{\delta}$ and $\varepsilon_{ij}=0,\pm 1$. 
\begin{lemma}\label{lem2.4} The deform initial data set $(\tilde{M}^n,\bar{g}_{\delta},\bar{k}_{\delta},\bar{E}_{\delta},\bar{B}_{\delta})$ has the following properties.
	\begin{enumerate}[label=(\alph*)]
		\item Each components of $\bar{k}_{\delta}$, $\bar{E}_{\delta}$, and $\bar{B}_{\delta}$ are $C^1$ path in $C^0(\Sigma)$ and $C^0$ path in $C^1(\Sigma)$. Moreover, $(\bar{k}_{\delta},\bar{E}_{\delta},\bar{B}_{\delta})$ agrees with $(\bar{\mathcal{K}},\bar{\mathcal{E}},\bar{\mathcal{B}})$ outside $(-\frac{\delta}{2},\frac{\delta}{2})$.
		\item $\bar{E}_{\delta}$, $\bar{B}_{\delta}$ and $\bar{k}_{\delta}$ are bounded by constants depends on $(\bar{\mathcal{K}},\bar{\mathcal{E}},\bar{\mathcal{B}})$, and $\gamma$ and not $\delta$.
	\end{enumerate}
\end{lemma}
\begin{proof}
	(a) Using definition of initial data set follows that the components $\bar{k}_{ij},\bar{E}_{i},\bar{B}_{ij}:(-\epsilon,\epsilon)\to C^{1-l}(\Sigma)$ are $C^{l}$ away from $\Sigma$ for $l=0,1$. Then, the components of deform initial data sets $\bar{k}_{\delta},\bar{E}_{\delta},\bar{B}_{\delta}:(-\epsilon,\epsilon)\to C^{1-l}(\Sigma)$ are $C^l$ away from $[-\frac{\delta^2}{100},\frac{\delta^2}{100}]$. For $t\in(-\frac{\delta}{4},\frac{\delta}{4})$, we have $\sigma_{\delta}(t)=\frac{\delta^2}{100}$ and we get standard mollification of potentials and metric which are smooth. Moreover, for $|t|>\frac{\delta}{2}$, we have $\sigma_{\delta}(t)=0$, $\gamma_{\delta}=\gamma$, $\psi_{\delta}=\psi$, $\chi_{\delta}=\chi$, and $\zeta_{\delta}=\zeta$. Therefore, $(\bar{k}_{\delta},\bar{E}_{\delta},\bar{B}_{\delta})$ agrees with $(\bar{\mathcal{K}},\bar{\mathcal{E}},\bar{\mathcal{B}})$.\\
\noindent(b) By Lemma \ref{lem1} and definition of $\bar{E}_{\delta}$, $\bar{B}_{\delta}$ and $\bar{k}_{\delta}$, we only need to show that derivative of deform potentials are bounded by constants depend on $(\bar{\mathcal{K}},\bar{\mathcal{E}},\bar{\mathcal{B}})$ and $\gamma$. Observe that
\begin{equation}
\partial_t\omega_{\delta}(t,\theta)=\int_{\mathbb{R}}\omega'\left(t-\sigma_{\delta}(t)s\right)\left(1-s\delta\sigma'(\frac{t}{\delta})\right)\phi(s) ds
\end{equation} for $\omega_{\delta}=\zeta_{\delta},\chi_{\delta},$ and $\psi_{\delta}$. Clearly this is bounded by constants depend on $(\bar{\mathcal{K}},\bar{\mathcal{E}},\bar{\mathcal{B}})$. 
\end{proof}
Next, we show behavior of energy density, energy flux, and Maxwell equations of deform initial data set and we prove the following generalization of \cite[Proposition 3.1]{miao2002positive} to axially symmetric initial data set.
\begin{prop}\label{propmain}
Let $(\tilde{M}^n,\bar{\mathcal{G}},\bar{\mathcal{K}},\bar{\mathcal{E}},\bar{\mathcal{B}})$ be a n-dimensional $(t-\phi^i)$-symmetric, simply connected initial data set admitting
corners across axially symmetric hypersurface $\Sigma$. Then there exists a family of 4-tuple $(\bar{g}_{\delta},\bar{k}_{\delta},\bar{E}_{\delta},\bar{B}_{\delta})$, where $\bar{g}_{\delta}$ is $C^2$ and $\bar{k}_{\delta},\bar{E}_{\delta},\bar{B}_{\delta}$ are $C^1$, for ${0<\delta\leq \delta_0}$ on $\tilde{M}^n$ so that $\bar{g}$ is uniformly close to $\bar{g}$ on $\tilde{M}^n$. Moreover, $(\bar{g}_{\delta},\bar{k}_{\delta},\bar{E}_{\delta},\bar{B}_{\delta})=(\bar{\mathcal{G}},\bar{\mathcal{K}},\bar{\mathcal{E}},\bar{\mathcal{B}})$ outside $\mathcal{O}_{\delta}=(-\frac{\delta}{2},\frac{\delta}{2})\times\Sigma$ and the energy density satisfies 
	\begin{equation}\label{mu1}
	\bar{\mu}_{\delta}(t,x)=O(1),\quad \text{for}\quad  (t,x)\in \left\{-\frac{\delta^2}{100}<|t|\leq \frac{\delta}{2}\right\}\times \Sigma
	\end{equation} and for $(t,x)\in [-\frac{\delta^2}{100},\frac{\delta^2}{100}]\times \Sigma$, we have
	\begin{equation}\label{mu2}
	\bar{\mu}_{\delta}(t,x)=O(1)+\left\{H(\Sigma,\bar{g}_-)(x)-H(\Sigma,\bar{g}_+)(x)\right\}\left\{\frac{100}{\delta^2}\phi(\frac{100t}{\delta^2})\right\},
	\end{equation}
	where $O(1)$ represents quantities that are bounded by constants depending only on $(\bar{\mathcal{G}},\bar{\mathcal{K}},\bar{\mathcal{E}},\bar{\mathcal{B}})$, but not on $\delta$. Moreover, for all $t\in \mathbb{R}$, we have $\text{tr}_{\bar{g}_{\delta}}\bar{k}_{\delta}(t)$, $\bar{J}_{\delta}(\eta_{(l)})(t)=0$, and
	\begin{equation}\label{fielddelta}
	 \text{div}_{\bar{g}_{\delta}}\bar{B}_{\delta}(t)=0,\qquad \text{div}_{\bar{g}_{\delta}}\bar{E}_{\delta}(t)=\frac{n-3}{\sqrt{3}}\star_{\bar{g}_{\delta}}\left(\bar{B}_{\delta}\wedge \bar{B}_{\delta}\right)(t),
	\end{equation} Furthermore, the angular momentum and charges of $(\tilde{M}^n,\bar{g}_{\delta},\bar{k}_{\delta},\bar{E}_{\delta},\bar{B}_{\delta})$ are conserved and equal to angular momentum and charges of $(\tilde{M}^n,\bar{\mathcal{G}},\bar{\mathcal{K}},\bar{\mathcal{E}},\bar{\mathcal{B}})$.
\end{prop}
\begin{proof}[Proof of Proposition \ref{propmain}] For $t\notin(-\frac{\delta}{2},\frac{\delta}{2})$ or outside $\mathcal{O}_{\delta}$, the initial data $(\bar{g}_{\delta},\bar{k}_{\delta},\bar{E}_{\delta},\bar{B}_{\delta})$ is same as $(\tilde{M}^n,\bar{\mathcal{G}},\bar{\mathcal{K}},\bar{\mathcal{E}},\bar{\mathcal{B}})$, thus the result hold. Let $t\in(-\epsilon,\epsilon)$. First, we show energy density $\bar{\mu}_{\delta}$ satisfies in equations \eqref{mu1} and \eqref{mu2}. The Gauss equation and evolution of mean curvature of $\Sigma$ leads to
	\begin{equation}\label{gauss}
	R_{\bar{g}_{\delta}}=\text{scal}_{\delta}-\left(H_{\delta}^2+|A_{\delta}|^2\right)-2D_{\nu}H_{\delta}
	\end{equation}where $\nu$ is unit normal vector field on $\Sigma$, $A_{\delta}$  and $H_{\delta}$ are the second fundamental form and mean curvature on $\Sigma$, and $\text{scal}_{\delta}$ is scalar curvature of hypersurface $\Sigma$. Then, substitution equation \eqref{gauss} in the Hamiltonian constraint equation \eqref{Hcons}, we obtain
	\begin{equation}
	16\pi \bar{\mu}_{\delta}=\text{scal}_{\delta}-\left(H_{\delta}^2+|A_{\delta}|^2\right)-2D_{\nu}H_{\delta}-|\bar{k}_{\delta}|_{\bar{g}_{\delta}}^2-\frac{2}{(n-2)^2}|\bar{E}_{\delta}|_{\bar{g}_{\delta}}^2-\frac{2}{(n-2)^3}|\bar{B}_{\delta}|_{\bar{g}_{\delta}}^2.
	\end{equation}The Lemma \ref{lem2.4} implies  $\bar{k}_{\delta}$, $\bar{B}_{\delta}$, and $\bar{E}_{\delta}$ are bounded by constants depends on $\bar{\mathcal{K}}$, $\bar{\mathcal{E}}$, $\bar{\mathcal{B}}$, and $\gamma$. Then, in light of \cite[Proposition 3.1]{miao2002positive}, we have \eqref{mu1} and \eqref{mu2}. Based on structure of $(t-\phi^i)$ symmetric data and Proposition \ref{proptphi}, the deform initial data set is maximal, $\bar{J}_{\delta}(\eta_{(l)})=0$, and is a solution of field equations in \eqref{fielddelta}. Finally, conservation of angular momentum and charges for any homologous surface to $S^{n-1}_{\infty}$ follows from field equations \eqref{fielddelta} and  \cite[Lemma 2.1]{dain2013lower} for 3-dimensional case and \cite[Section 4]{alaee2017relating} for 4-dimensional case.
	\end{proof}
\section{Proof of Main Results}\label{sec3}
In order to prove the main theorem, we need to show that under appropriate conformal transformation the energy density of the deform initial data set is non-negative and conservation of angular momentum and charges hold. Let $C$, $C_1$, $C_2$, and $C_3$ be constants depend on initial data set $(\bar{\mathcal{G}},\bar{\mathcal{K}},\bar{\mathcal{E}},\bar{\mathcal{B}})$. Consider the following jump on the mean curvatures of an axially symmetric hypersurface $\Sigma$ within $(t-\phi^i)$-symmetric data set 
\begin{equation}\label{meancurv}
	H_-(\Sigma,\bar{g}_-)\geq H_+(\Sigma,\bar{g}_+).
\end{equation}By Remark \ref{remarkmean} and definition of $(t-\phi^i)$-symmetric data set,  this is equivalent to $H_-(\Sigma,{g}_-)\geq H_+(\Sigma,{g}_+)$ for related axially symmetric data set. Assume the initial data set has one designated asymptotically flat $\infty$ and the other either asymptotically cylindrical $\infty_{1}$ or boundary $\partial\tilde{M}^n$. Then we have the following proposition. 
\begin{proposition}\label{Prop3.1}
	Let $\bar{g}_{\delta}$ be a $C^2$ $(t-\phi^i)$-symmetric asymptotically flat metric with one designated asymptotically flat $\infty$ and the other either asymptotically cylindrical $\infty_1$ or boundary $\partial\tilde{M}^n$, and $c_n=\frac{n-2}{4(n-1)}$, then \begin{equation}\label{PDEu}
	\begin{cases}
	\Delta_{\bar{g}_{\delta}}u_{\delta}+c_n\bar{\mu}_{\delta-}u_{\delta}=0 & \text{on}\quad \tilde{M}^n\\
	u_{\delta}=1 & \text{on}\quad \infty\\
	u_{\delta}=1& \text{on}\quad \infty_1
	\end{cases}\quad \text{and}\quad
	\begin{cases}
	\Delta_{\bar{g}_{\delta}}u_{\delta}+c_n\bar{\mu}_{\delta-}u_{\delta}=0 & \text{on}\quad \tilde{M}^n\\
	u_{\delta}=1 & \text{on}\quad \infty\\
	\frac{\partial u_{\delta}}{\partial n}=0& \text{on}\quad \partial\tilde{M}^n
	\end{cases}
	\end{equation}have a unique axially symmetric $C^2$ solution $u_{\delta}\geq 1$ on $\tilde{M}^n$ so that, $\lim_{\delta\to 0}||u_{\delta}-1||_{L^{\infty}(\tilde{M}^n)}=0$, $u_{\delta}=1+\frac{A_{\delta}}{|x|^{n-2}}+O_1(|x|^{1-n})$ for some constant $A_{\delta}$ as $|x|\to\infty$, and $\frac{\partial u_{\delta}}{\partial n}=0$ at $\infty_1$.
\end{proposition}
\begin{proof}By equations \eqref{mu1} and \eqref{mu2} for $\bar{\mu}_{\delta-}$ and jump in the mean curvature in \eqref{meancurv}, we have
	\begin{equation}\label{eq3}
	\begin{cases}
	\bar{\mu}_{\delta-}=0&\text{outside $\mathcal{O}_{\delta}$}\\
	|\bar{\mu}_{\delta-}|\leq C_1&\text{inside $\mathcal{O}_{\delta}$}.\\
	\end{cases}
	\end{equation}Since outside a compact set $\bar{\mu}_{\delta-}=0$ and the first eigenvalue of Laplacian is zero on asymptotically cylindrical end $\infty_1$, the solution of \eqref{PDEu} can have asymptotically constant behavior at asymptotically cylindrical end $\infty_1$. Moreover, asymptotically cylindrical end yields to exponential decay for solutions and we obtain $\frac{\partial u_{\delta}}{\partial n}=0$ at $\infty_1$. Then using \cite[Lemma 3.3]{schoen1979proof}, we get existence of a unique $C^2$ positive solution with asymptotic $u_{\delta}=1+\frac{A_{\delta}}{|x|^{n-2}}+O(|x|^{1-n})$ at $\infty$ such that 
	\begin{equation}
	A_{\delta}=\frac{1}{\omega_n}\lim_{\delta\to 0}\int_{\tilde{M}}\left(-|\nabla_{\bar{g}_{\delta}}u_{\delta}|^2+c_n\bar{\mu}_{\delta-}u_{\delta}^2\right) dx_{\bar{g}_{\delta}},
	\end{equation}where $dx_{\bar{g}_{\delta}}$ is volume form with respect to $\bar{g}_{\delta}$ and $\omega_n$ is volume of unit $(n-1)$-dimensional unit sphere. Since $u_{\delta}$ is superharmonic, by strong maximum principle the minimum is at ends or boundary. For Dirichlet problem in \eqref{PDEu}, it is clear that $u_{\delta}\geq 1$. For other problem, if the minimum is at asymptotically flat end $\infty$, then $u_{\delta}\geq 1$. But if the minimum is at $\partial\tilde{M}^n$ and less than one, then using Hopf maximum principle we should have $\nu(u_{\delta})<0$, which is a contradiction. Therefore, $u_{\delta}\geq 1$. Finally, using \cite[Proposition 4.1]{miao2002positive}, we have $\lim_{\delta\to 0}||u_{\delta}-1||_{L^{\infty}(\tilde{M}^n)}=0$.
\end{proof}
Next we define the following conformal transformation
\begin{equation}\label{conf}
\tilde{g}_{\delta}=u^{\frac{4}{n-2}}_{\delta}\bar{g}_{\delta},\qquad \tilde{k}_{\delta}=u^{-2}_{\delta}\bar{k}_{\delta},\qquad \tilde{E}_{\delta}=u^{-2}_{\delta}\bar{E}_{\delta}\qquad \tilde{B}_{\delta}=u^{-2}_{\delta}\bar{B}_{\delta}
\end{equation}where $\bar{g}_{\delta}$ and $\bar{k}_{\delta}$ are $(0,2)$-tensor, $\bar{E}_{\delta}$ is a 1-form, and $\bar{B}_{\delta}$ is a $(n-2)$-form. Then we have 
\begin{proposition}\label{prop3.2}Consider the conformal initial data set $(\tilde{g}_{\delta},\tilde{k}_{\delta},\tilde{E}_{\delta},\tilde{B}_{\delta})$ in equation \eqref{conf}. Then the metric is $C^2$, $\tilde{k}_{\delta},\tilde{E}_{\delta}$, and $\tilde{B}_{\delta}$ are $C^1$, and it has non-negative energy density and vanishing energy flux in the direction of symmetries. Furthermore, its mass converges to the mass of $(\bar{\mathcal{G}},\bar{\mathcal{K}},\bar{\mathcal{E}},\bar{\mathcal{B}})$ and angular momentum and charges are the same as angular momentum and charges of $(\bar{\mathcal{G}},\bar{\mathcal{K}},\bar{\mathcal{E}},\bar{\mathcal{B}})$. 
\end{proposition}
\begin{proof}It follows from Proposition \ref{propmain} and Proposition \ref{Prop3.1} that the metric $\tilde{g}_{\delta}$ is $C^2$ and $(\tilde{k}_{\delta},\tilde{E}_{\delta},\tilde{B}_{\delta})$ are $C^1$. Assuming $n=3,4$, we have 
	\begin{equation}
	\begin{split}
	\tilde{\mu}_{\delta}&=R(\tilde{g}_{\delta})-|\tilde{k}_{\delta}|^2_{\tilde{g}_{\delta}}-\frac{2}{(n-2)^2}|\tilde{E}_{\delta}|^2_{\tilde{g}_{\delta}}-\frac{2}{(n-2)^3}|\tilde{B}_{\delta}|^2_{\tilde{g}_{\delta}}\\
	&=u^{\frac{4}{2-n}}_{\delta}\left(\bar{\mu}_{\delta+}+\left(1-u_{\delta}^{4\left(\frac{n-1}{2-n}\right)}\right)|\bar{k}_{\delta}|^2_{\bar{g}_{\delta}}+\frac{2}{(n-2)^2}\left(1-u_{\delta}^{-4}\right)|\bar{E}_{\delta}|^2_{\bar{g}_{\delta}}\right.\\
	&\left.\qquad\qquad+\frac{2}{(n-2)^3}\left(1-u_{\delta}^{2(n-6)}\right)|\bar{B}_{\delta}|^2_{\bar{g}_{\delta}}\right).
	\end{split}
	\end{equation}Since $u_{\delta}\geq 1$, we have  $\tilde{\mu}_{\delta}\geq 0$. By asymptotic of $u_{\delta}$ we have
	\begin{equation}
	\begin{split}
	m_{ADM}(\tilde{g}_{\delta})&=\frac{1}{16\pi}\lim_{r\to\infty}\int_{S^2_{r}}\left((\tilde{g}_{\delta})_{ij,i}-(\tilde{g}_{\delta})_{ii,j}\right)\tilde{\nu}^jd\tilde{S}_{\delta}\\
	&=	\frac{1}{16\pi}\lim_{r\to\infty}\int_{S^2_{r}}\left((\bar{g}_{\delta})_{ij,i}-(\bar{g}_{\delta})_{ii,j}\right){\nu}^jd{S}_{\delta}-\frac{1}{2\pi(n-2)}\lim_{r\to\infty}\int_{S^2_{r}}(u_{\delta})_{,j}{\nu}^jd{S}_{\delta},\\
	&=m_{ADM}(\bar{g}_{\delta})+\frac{2}{n-2}A_{\delta}
	\end{split}
	\end{equation}where $\tilde{\nu}=u_{\delta}^{-\frac{2}{n-2}}\nu$ and $d\tilde{S}_{\delta}=u_{\delta}^{\frac{2(n-1)}{n-2}}d{S}_{\delta}$. Then using equation \eqref{PDEu} and H\"older inequality we have
	\begin{equation}
	\begin{split}
	\lim_{\delta\to 0} A_{\delta}&\leq C\lim_{\delta\to 0}\int_{\tilde{M}}|\nabla_{\bar{g}_{\delta}}u_{\delta}|^2dx_{\bar{g}_{\delta}}+C\lim_{\delta\to 0}\left(\int_{\mathcal{O}_{\delta}}|\bar{\mu}_{\delta-}|^{n/2}dx_{\bar{g}_{\delta}}\right)^{2/n}\left(\int_{\mathcal{O}_{\delta}}u_{\delta}^{\frac{2n}{n-2}}dx_{\bar{g}_{\delta}}\right)^{\frac{n-2}{2n}}\\
	&\leq C\lim_{\delta\to 0}\left(\int_{\tilde{M}}|\bar{\mu}_{\delta-}|^{n/2}dx_{\bar{g}_{\delta}}\right)^{2/n}\left(\int_{\tilde{M}}|u_{\delta}-1|^{\frac{2n}{n-2}}dx_{\bar{g}_{\delta}}\right)^{\frac{n-2}{2n}}\\
	&\quad+
	C\lim_{\delta\to 0}\left(\int_{\tilde{M}}|\bar{\mu}_{\delta-}|^{\frac{2n}{n+2}}dx_{\bar{g}_{\delta}}\right)^{\frac{n+2}{2n}}\left(\int_{\tilde{M}}|u_{\delta}-1|^{\frac{2n}{n-2}}dx_{\bar{g}_{\delta}}\right)^{\frac{n-2}{n}}\\
	&\quad+C\lim_{\delta\to 0}\left(\int_{\mathcal{O}_{\delta}}|\bar{\mu}_{\delta-}|^{n/2}dx_{\bar{g}_{\delta}}\right)^{2/n}\left(\int_{\mathcal{O}_{\delta}}u_{\delta}^{\frac{2n}{n-2}}dx_{\bar{g}_{\delta}}\right)^{\frac{n-2}{2n}}.
	\end{split}
	\end{equation}Combining this with equation \eqref{eq3} and Proposition \ref{Prop3.1}, we get $\lim_{\delta\to 0}A_{\delta}=0$. Therefore,  
	\begin{equation}
	\lim_{\delta\to0}m_{ADM}(\tilde{g}_{\delta})=m_{ADM}(\bar{\mathcal{G}})
	\end{equation}Next, observe that 
	\begin{equation}
	\text{div}_{\tilde{g}_{\delta}}\tilde{k}_{\delta}=u^{-\frac{2(n-4)}{n-2}}\text{div}_{\bar{g}_{\delta}}\bar{k}_{\delta},\,\,\,\text{}\text{}	\text{div}_{\tilde{g}_{\delta}}\tilde{B}_{\delta}=u^{-\frac{2(n-4)}{n-2}}\text{div}_{\bar{g}_{\delta}}\bar{B}_{\delta},\,\,\,	\text{div}_{\tilde{g}_{\delta}}\tilde{E}_{\delta}=u^{-\frac{2n}{n-2}}\text{div}_{\bar{g}_{\delta}}\bar{E}_{\delta},
	\end{equation}and $\star_{\tilde{g}_{\delta}}\alpha=u^{\frac{2n-4p}{n-2}}\star_{\bar{g}_{\delta}}\alpha$ for a $p$-form $\alpha$. Therefore, we have 
	\begin{equation}
	\tilde{J}_{\delta}=u^{-\frac{2(n-4)}{n-2}}\bar{J}_{\delta},\quad \text{div}_{\tilde{g}_{\delta}}\tilde{B}_{\delta}=0,\quad \text{div}_{\tilde{g}_{\delta}}\tilde{E}_{\delta}=\frac{n-3}{\sqrt{3}}\star_{\tilde{g}_{\delta}}\left(\tilde{B}_{\delta}\wedge \tilde{B}_{\delta}\right).
	\end{equation}Hence $\tilde{J}_{\delta}(\eta_{(l)})=0$ and it shows that angular momentum and charges are conserved as in Proposition \ref{propmain}.
\end{proof}
Now we have all tools to prove Theorem \ref{thm1.2}. 
\begin{proof}[Proof of Theorem \ref{thm1.2}]First we prove the strict inequality and then the rigidity cases. From Proposition \ref{proptphi}, the axially symmetric initial data set $(\mathcal{G},\mathcal{K},\mathcal{E},\mathcal{B})$ has same mass, angular momentum and charges as related $(t-\phi^i)$-symmetric data set $(\bar{\mathcal{G}},\bar{\mathcal{K}},\bar{\mathcal{E}},\bar{\mathcal{B}})$. Assume $(\bar{\mathcal{G}},\bar{\mathcal{K}},\bar{\mathcal{E}},\bar{\mathcal{B}})$ has an outermost minimal surface $\Sigma_{\text{min}}$ that enclosed by $\Sigma$. Then by \cite[Proposition 3.1.]{bryden2018positive}, $\Sigma_{\text{min}}$ is axially symmetric. We cut the initial data from $\Sigma_{\text{min}}$ and consider $(\bar{\mathcal{G}},\bar{\mathcal{K}},\bar{\mathcal{E}},\bar{\mathcal{B}})$ with outermost minimal surface boundary $\Sigma_{\text{min}}$. By proposition \ref{Prop3.1}, the initial data $(\tilde{g}_{\delta},\tilde{k}_{\delta},\tilde{E}_{\delta},\tilde{B}_{\delta})$ has same angular momentum and charges as $(\bar{\mathcal{G}},\bar{\mathcal{K}},\bar{\mathcal{E}},\bar{\mathcal{B}})$ with axially symmetric minimal surface boundary $\Sigma_{\text{min}}$. 
	
For parts (a) and (b), since $\Sigma_{\text{min}}$ is minimal surface boundary, we double the initial data set. Then the mean curvature of two sides coincide and the metric is Lipschitz across $\Sigma_{\text{min}}$. Moreover, mass, angular momentum, and charges are conserved through this doubling. Then we deform again the initial data set and obtain a complete initial data set with two asymptotically flat ends and denote it by $(\tilde{g}_{\delta},\tilde{k}_{\delta},\tilde{E}_{\delta},\tilde{B}_{\delta})$. Therefore, by \cite{costa2010proof,alaee2017relating,alaee2015proof,alaee2017mass}, since $\tilde{g}_{\delta}$ is $C^2$, $\tilde{k}_{\delta},\tilde{E}_{\delta}$, and $\tilde{B}_{\delta}$ are $C^1$, and it has non-negative energy density and vanishing energy flux in the direction of symmetries, the mass, angular momentum, and charge inequality hold for initial data set $(\tilde{g}_{\delta},\tilde{k}_{\delta},\tilde{E}_{\delta},\tilde{B}_{\delta})$, that is
	\begin{equation}
	m_{ADM}(\tilde{g}_{\delta})^2> \frac{Q^2+\sqrt{Q^4+4\mathcal{J}^2}}{2},\qquad \text{for $n=3$},
	\end{equation}if $M^4=\mathbb{R}^3\times S^3$, we have
	\begin{equation}
		m_{ADM}(\tilde{g}_{\delta})> \frac{27\pi}{8}\frac{\left(\mathcal{J}_1+\mathcal{J}_2\right)^2}{\left(2m_{ADM}(\tilde{g}_{\delta})+\sqrt{3}|Q|\right)^2}+\sqrt{3}|Q|,\qquad \text{for $n=4$},
	\end{equation}Note that the inequality is not sharp because the initial data set has two asymptotically flat ends. For part (c), we do not need to double the initial data set and we can apply \cite{alaee2017mass} with topology $M^4=S^2\times B^2\# \mathbb{R}^4$ and we obtain
	\begin{equation}
	m_{ADM}(\tilde{g}_{\delta})^3> \frac{27\pi}{4}|\mathcal{J}_2||\mathcal{J}_1-\mathcal{J}_2|,\qquad \text{for $n=4$}.
	\end{equation}It follows from Proposition \ref{prop3.2} that
	\begin{equation}
	m_{ADM}(\mathcal{G})=m_{ADM}(\bar{\mathcal{G}})=\lim_{\delta\to 0}m_{ADM}(g_{\delta})=\lim_{\delta\to 0}m_{ADM}(\tilde{g}_{\delta}).
	\end{equation}Therefore, the strict inequalities hold for initial data set $(\mathcal{G},\mathcal{K},\mathcal{E},\mathcal{B})$. 
	
For the rigidity case, the initial data set has asymptotically cylindrical end $\infty_1$. We establish the rigidity for 3-dimensional initial data sets and 4-dimensional cases are similar. Assume rigidity in \eqref{in3d}, then
\begin{equation}\label{rig}
\tfrac{Q^2+\sqrt{Q^4+4\mathcal{J}^2}}{2}=m_{ADM}({\mathcal{G}})^2=m_{ADM}(\bar{\mathcal{G}})^2=\lim_{\delta\to 0}m_{ADM}(g_{\delta})^2=\lim_{\delta\to 0}m_{ADM}(\tilde{g}_{\delta})^2.
\end{equation}The deform data $(\bar{g}_{\delta},\bar{k}_{\delta},\bar{E}_{\delta},\bar{B}_{\delta})$ agrees with $(\bar{\mathcal{G}},\bar{\mathcal{K}},\bar{\mathcal{E}},\bar{\mathcal{B}})$ outside small neighborhood $\mathcal{O}_{\delta}$ of $\Sigma$ according to Proposition \ref{propmain}. Since the metric  $\bar{g}_{\delta}$ is $C^2$, there exist a global coordinate system $(\rho,z,\phi)$ \cite{chrusciel2008mass} such that it  depends on $\delta$ and the conformal metric $\tilde{g}_{\delta}=u_{\delta}^4\bar{g}_{\delta}$ takes the following form
\begin{equation}
\tilde{g}_{\delta}=u_{\delta}^4e^{-2{U}_{\delta}+2\alpha_{\delta}}\left(d\rho^2+dz^2\right)+\rho^2u_{\delta}^4e^{-2\bar{U}_{\delta}}d\phi^2,
\end{equation}where $\rho\in[0,\infty)$, $z\in\mathbb{R}$, and $\phi\in[0,2\pi)$. Moreover, it has the following fall-off at infinity
\begin{equation}
{U}_{\delta}=O(r^{-1/2-\kappa}),\qquad \alpha_{\delta}=O(r^{-1/2-\kappa}), \qquad \text{for $\kappa>0$}.
\end{equation}Furthermore, $\alpha_{\delta}=0$ on the axis $\Gamma=\{\rho=0\}$. Set ${V}_{\delta}={U}_{\delta}-2\log u_{\delta}$. The scalar curvature of the metric has the following simple expression
\begin{equation}\label{scal}
2e^{-2V_{\delta}+2\alpha_{\delta}}R(\tilde{g}_{\delta})=8\Delta V_{\delta}-4\Delta_{\rho,z}\alpha_{\delta}-4|\nabla V_{\delta}|^2,
\end{equation}where $\Delta$ is the Laplacian with respect to flat metric $\delta_3=d\rho^2+dz^2+\rho^2d\phi^2$ on $\mathbb{R}^3$ and $\Delta_{\rho,z}$ is Laplacian with respect to flat metric $\delta_2=d\rho^2+dz^2$ on $\mathbb{R}^2$. Therefore, using the Hamiltonian constraint equation \eqref{Hcons} and maximality condition $\text{tr}_{\tilde{g}_{\delta}}\tilde{k}_{\delta}=0$, the ADM mass has the following simple expression \cite{khuri2016positive} 
\begin{equation}\label{m1}
m_{ADM}(\tilde{g}_{\delta})=I(\Psi_{\delta})+\int_{\mathbb{R}^3}e^{-2V_{\delta}+2\alpha_{\delta}}\tilde{\mu}_{\delta}dx,
\end{equation} where $\Psi_{\delta}=(V_{\delta},\zeta_{\delta},\chi_{\delta},\psi_{\delta})$, $dx$ is the volume form with respect to $\delta_3$, and $I(\Psi_{\delta})$ is the reduced harmonic energy from $\mathbb{R}^3\backslash\Gamma\to \mathbb{H}^2_{\mathbb{C}}$, that is 
\begin{equation}
I(\Psi_{\delta})=\frac{1}{8\pi}\int_{\mathbb{R}^3}\left(|\nabla V_{\delta}|^2+\frac{e^{4V_{\delta}}}{\rho^4}|\nabla \zeta_{\delta}+\chi_{\delta}\nabla\psi_{\delta}-\psi_{\delta}\nabla\chi_{\delta}|^2+\frac{e^{2V_{\delta}}}{\rho^4}\left(|\nabla\chi_{\delta}|^2+|\nabla\psi_{\delta}|^2\right)\right)dx.
\end{equation}Assume $\Psi_0=(U_{0},\zeta_{0},\chi_{0},\psi_{0})$ is the harmonic map of the canonical slice of extreme Kerr-Newman black hole $(g_0,k_0,E_0,B_0)$ such that $I(\Psi_0)=\tfrac{Q^2+\sqrt{Q^4+4\mathcal{J}^2}}{2}$. Then, applying the Schoen and Zhou gap inequality \cite{schoen2013convexity}, we obtain 
\begin{equation}\label{m2}
I(\Psi_{\delta})-I(\Psi_{0})\geq C\left(\int_{\mathbb{R}^3}\text{dist}^6_{\mathbb{H}^2_{\mathbb{C}}}\left(\Psi_{\delta},\Psi_{0}\right)dx\right)^{1/3},
\end{equation}Combing this with equations \eqref{rig} and \eqref{m1} shows that
\begin{equation}\label{m3}
\begin{split}
0&=\lim_{\delta\to 0}\left(m_{ADM}(\tilde{g}_{\delta})^2-\tfrac{Q^2+\sqrt{Q^4+4\mathcal{J}^2}}{2}\right)\\
&\geq \lim_{\delta\to 0}\left(I(\Psi_{\delta})^2-\tfrac{Q^2+\sqrt{Q^4+4\mathcal{J}^2}}{2}\right)+\lim_{\delta\to 0}\left(\int_{\mathbb{R}^3}e^{-2V_{\delta}+2\alpha_{\delta}}\tilde{\mu}_{\delta}dx\right)^2\\
&\geq C\lim_{\delta\to 0}\left(\int_{\mathbb{R}^3}\text{dist}^6_{\mathbb{H}^2_{\mathbb{C}}}\left(\Psi_{\delta},\Psi_{0}\right)dx\right)^{2/3}+\lim_{\delta\to 0}\left(\int_{\mathbb{R}^3}e^{-2V_{\delta}+2\alpha_{\delta}}\tilde{\mu}_{\delta}dx\right)^2.
\end{split}
\end{equation}It follows that 
\begin{equation}\label{eqlim}
\lim_{\delta\to 0}\int_{\mathbb{R}^3}\text{dist}^6_{\mathbb{H}^2_{\mathbb{C}}}\left(\Psi_{\delta},\Psi_{0}\right)dx=0,\qquad \lim_{\delta\to 0}\int_{\mathbb{R}^3}e^{-2V_{\delta}+2\alpha_{\delta}}\tilde{\mu}_{\delta}dx=0.
\end{equation} By Remark \ref{rem2} and Proposition \ref{proptphi}, $\Psi=(U,\zeta,\chi,\psi)$ and $\alpha$ characterize $(\bar{\mathcal{G}},\bar{\mathcal{K}},\bar{\mathcal{E}},\bar{\mathcal{B}})$. Next we show that $\Psi=\Psi_0$. Observe that
\begin{equation}
\begin{split}\label{eq2}
&\left(\text{dist}_{\mathbb{H}^2_{\mathbb{C}}}\left(\Psi_{\delta},\Psi\right)-\text{dist}_{\mathbb{H}^2_{\mathbb{C}}}\left(\Psi,\Psi_{0}\right)\right)^6-\text{dist}^6_{\mathbb{H}^2_{\mathbb{C}}}\left(\Psi,\Psi_{0}\right)\\
&\leq \text{dist}^6_{\mathbb{H}^2_{\mathbb{C}}}\left(\Psi_{\delta},\Psi_{0}\right)-\text{dist}^6_{\mathbb{H}^2_{\mathbb{C}}}\left(\Psi,\Psi_{0}\right)\\
&\leq\left(\text{dist}_{\mathbb{H}^2_{\mathbb{C}}}\left(\Psi_{\delta},\Psi\right)+\text{dist}_{\mathbb{H}^2_{\mathbb{C}}}\left(\Psi,\Psi_{0}\right)\right)^6-\text{dist}^6_{\mathbb{H}^2_{\mathbb{C}}}\left(\Psi,\Psi_{0}\right).
\end{split}
\end{equation}If we show that
\begin{equation}\label{eq4}
\lim_{\delta\to 0}\int_{\mathbb{R}^3}\text{dist}^6_{\mathbb{H}^2_{\mathbb{C}}}\left(\Psi_{\delta},\Psi\right)dx=0,
\end{equation}then in light of equation \eqref{eq2}, we can pass the limit through integral
\begin{equation}\label{eq4.25}
0=\lim_{\delta\to 0}\int_{\mathbb{R}^3}\text{dist}^6_{\mathbb{H}^2_{\mathbb{C}}}\left(\Psi_{\delta},\Psi_{0}\right)dx=\int_{\mathbb{R}^3}\text{dist}^6_{\mathbb{H}^2_{\mathbb{C}}}\left(\Psi,\Psi_{0}\right)dx=0.
\end{equation}Using the triangle inequality and $u_{\delta}\geq 1$, a direct computation produces
\begin{equation}\label{dist}
\begin{split}
&\text{dist}_{\mathbb{H}^2_{\mathbb{C}}}\left(\Psi_{\delta},\Psi\right)\\
&\leq \text{dist}_{\mathbb{H}^2_{\mathbb{C}}}\left(({V}_{\delta},\zeta_{\delta},\chi_{\delta},\psi_{\delta}),({U},\zeta_{\delta},\chi_{\delta},\psi_{\delta})\right)+\text{dist}_{\mathbb{H}^2_{\mathbb{C}}}\left(({U},\zeta_{\delta},\chi_{\delta},\psi_{\delta}),({U},\zeta,\chi_{\delta},\psi_{\delta})\right)\\
&\quad+\text{dist}_{\mathbb{H}^2_{\mathbb{C}}}\left(({U},\zeta,\chi_{\delta},\psi_{\delta}),({U},\zeta,\chi,\psi_{\delta})\right)+\text{dist}_{\mathbb{H}^2_{\mathbb{C}}}\left(({U},\zeta,\chi,\psi_{\delta}),({U},\zeta,\chi,\psi)\right)\\
  &\leq C\left(|U-V_{\delta}|+\frac{e^{2U}}{\rho^2}\left(|\zeta-\zeta_{\delta}|+|\psi_{\delta}||\chi-\chi_{\delta}|+|\chi||\chi-\chi_{\delta}|\right)\right.\\
&\left.\quad+\frac{e^{U}}{\rho}\left(|\chi-\chi_{\delta}|+|\psi-\psi_{\delta}|\right)\right).
\end{split}
\end{equation}Since $({U}_{\delta},\zeta_{\delta},\chi_{\delta},\psi_{\delta})=({U},\zeta,\chi,\psi)$ outside $\mathcal{O}_{\delta}$ by Proposition \ref{propmain}, we have
 \begin{equation}\label{eq423}
 \begin{split}
 \int_{\mathbb{R}^3}\text{dist}^6_{\mathbb{H}^2_{\mathbb{C}}}\left(\Psi_{\delta},\Psi\right) dx &\leq C\int_{\mathbb{R}^3}|\log{u}_{\delta}|^6dx+C\int_{\mathcal{O}_{\delta}}|U-U_{\delta}|^6dx\\
 &\quad+C\int_{\mathcal{O}_{\delta}}\frac{e^{12U}}{\rho^{12}}(|\zeta-\zeta_{\delta}|^{6}+|\psi_{\delta}|^6|\chi-\chi_{\delta}|^6+|\chi|^6|\chi-\chi_{\delta}|^6)dx\\
 &\quad+C\int_{\mathcal{O}_{\delta}}\frac{e^{6U}}{\rho^6}(|\chi-\chi_{\delta}|^6+|\psi-\psi_{\delta}|^6)dx,
 \end{split}
 \end{equation}By Lemma \ref{lem2.4} and Lemma \ref{lem1}, as $\delta\to 0$, all terms converge to $0$ except the first term. But, since $u_{\delta}\geq 1$ is a solution of \eqref{PDEu} and using the Sobolev inequality and H\"older inequality similar to \cite[Proposition 4.1]{miao2002positive}, we have
 \begin{equation}\label{eq424}
 \begin{split}
 C_2 \int_{\mathbb{R}^3}|\log{u}_{\delta}|^6dx&\leq\int_{\mathbb{R}^3}e^{-3{U}_{\delta}+{\alpha}_{\delta}}|\log{u}_{\delta}|^6dx=\int_{\tilde{M}^3}|\log{u}_{\delta}|^6dx_{\bar{g}_{\delta}}\\
 &\leq \int_{\tilde{M}^3}|{u}_{\delta}-1|^6dx_{\bar{g}_{\delta}}\leq C\left(\int_{\tilde{M}^3}|\bar{\mu}_{\delta-}|^{6/5} dx_{\bar{g}_{\delta}}\right)^{5}\to 0,
 \end{split}
 \end{equation}as $\delta\to 0$. Combining this with \eqref{eq423} and \eqref{eq4}, it follows that $\Psi=\Psi_0$. Next, using equation \eqref{eqlim} and ${V}_{\delta}={U}_{\delta}-2\log u_{\delta}$, we have 
 \begin{equation}\label{lim}
 \begin{split}
 0&=\lim_{\delta\to 0}\int_{\mathbb{R}^3}e^{-2V_{\delta}+2\alpha_{\delta}}\tilde{\mu}_{\delta}dx\geq \lim_{\delta\to 0}\int_{\mathbb{R}^3}e^{-2U_{\delta}+2\alpha_{\delta}}\bar{\mu}_{\delta+}dx=\lim_{\delta\to 0}\int_{\tilde{M}^3}e^{U_{\delta}}\bar{\mu}_{\delta+}dx_{\bar{g}_{\delta}}\\
 &=\lim_{\delta\to 0}\int_{\mathcal{O}_{\delta}}e^{U_{\delta}}\bar{\mu}_{\delta+}dx_{\bar{g}_{\delta}}+\lim_{\delta\to 0}\int_{\tilde{M}^3\backslash\mathcal{O}_{\delta}}e^{U}\bar{\mu}dx_{\bar{g}}\geq 0,
 \end{split}
 \end{equation}where the last equality follows from $\bar{g}_{\delta}=\bar{g}$ and $\bar{\mu}_{\delta+}=\bar{\mu}$ outside $\mathcal{O}_{\delta}$. This shows that $\bar{\mu}=0$ away from $\Sigma$.
 
 We claim that if there exist a strict jump of mean curvature across $\Sigma$, that is $H_+(\Sigma,\bar{g}_+)(x)>H_-(\Sigma,\bar{g}_-)(x)$ for some $x\in\Sigma$, we get a contradiction to \eqref{lim}. Assume there is a strict jump of mean curvature across $\Sigma$. Since mean curvatures are continuous functions on $\Sigma$, there exist a compact set $\mathcal{C}\subset \Sigma$ such that 
 \begin{equation}
 H_+(\Sigma,\bar{g}_+)(x)-H_-(\Sigma,\bar{g}_-)(x)\geq \beta,\qquad \forall x\in\mathcal{C},
 \end{equation}for some constant $\beta>0$. Then by Proposition \ref{propmain} and equation \eqref{eq3}, we have 
 \begin{equation}
 \bar{\mu}_{\delta+}(t,x)\geq \beta\left\{\tfrac{100}{\delta^2}\phi(\tfrac{100t}{\delta^2})\right\}- C_1,\qquad  \forall (t,x)\in [-\tfrac{\delta^2}{100},\tfrac{\delta^2}{100}]\times \mathcal{C}\subset\mathcal{O}_{\delta}.
 \end{equation}which suggests that the energy density of $(\bar{g}_{\delta},\bar{k}_{\delta},\bar{E}_{\delta},\bar{B}_{\delta})$ and  $(\tilde{g}_{\delta},\tilde{k}_{\delta},\tilde{E}_{\delta},\tilde{B}_{\delta})$ has a fixed amount of concentration on $\mathcal{C}$. Therefore, we get the following estimate which is a contradiction to equation \eqref{lim} 
 \begin{equation}
 \begin{split}
 \lim_{\delta\to 0}\int_{\mathcal{O}_{\delta}}e^{U_{\delta}}\bar{\mu}_{\delta+}dx_{\bar{g}_{\delta}}\geq \beta C_3|\mathcal{C}|>0,
 \end{split}
 \end{equation}where $|\mathcal{C}|$ is the measure of compact set $\mathcal{C}$. Hence,  $H_+(\Sigma,\bar{g}_+)(x)=H_-(\Sigma,\bar{g}_-)(x)$ for all $x\in\Sigma$. Next, we show that $\alpha=\alpha_0$. By constraint equation for the initial data $(\bar{\mathcal{G}},\bar{\mathcal{K}},\bar{\mathcal{E}},\bar{\mathcal{B}})$ on $\tilde{M}^3\backslash\bar{\Omega}$ and $\bar{\mu}=0$ away from $\Sigma$, we have
 \begin{equation}\label{alpha1}
 \begin{split}
R(\bar{\mathcal{G}})&=|\bar{\mathcal{K}}|^2_{\bar{\mathcal{G}}}+2|\bar{\mathcal{E}}|^2_{\bar{\mathcal{G}}}+2|\bar{\mathcal{B}}|^2_{\bar{\mathcal{G}}}\\
 &=2\frac{e^{6U_{+}-2\alpha_{+}}}{\rho^4}|\nabla \zeta_{+}+\chi_{+}\nabla\psi_{+}-\psi_{+}\nabla\chi_{+}|^2+
 2\frac{e^{4U_{+}-2\alpha_{+}}}{\rho^4}\left(|\nabla\chi_{+}|^2+|\nabla\psi_{+}|^2\right)\\
 &=2\frac{e^{6U_0-2\alpha_{+}}}{\rho^4}|\nabla \zeta_0+\chi_0\nabla\psi_{0}-\psi_{0}\nabla\chi_{0}|^2+
 2\frac{e^{4U_{0}-2\alpha_{+}}}{\rho^4}\left(|\nabla\chi_{0}|^2+|\nabla\psi_{0}|^2\right)\\
 &=e^{2\alpha_0-2\alpha_{+}}R({g}_0)
 \end{split}
 \end{equation}Then scalar curvature equation \eqref{scal} becomes
 \begin{equation}\label{alpha2}
 \begin{split}
2e^{-2U_0+2\alpha_{+}}R(\bar{\mathcal{G}})&=8\Delta U_{+}-4\Delta_{\rho,z}\alpha_{+}-4|\nabla U_{+}|^2\\
 &=8\Delta U_0-4\Delta_{\rho,z}\alpha_{+}-4|\nabla U_0|^2\\
 &=2e^{-2U_0+2\alpha_{0}}R({g}_0)+4\Delta_{\rho,z}\left(\alpha_0-\alpha_+\right)
 \end{split}
 \end{equation}Combining this with \eqref{alpha1} yields $\Delta_{\rho,z}\left(\alpha_0-\alpha_{+}\right)=0$ on $\tilde{M}^3\backslash\bar{\Omega}$. Multiplying it to $\left(\alpha_0-\alpha_{+}\right)$ and integrating over $O_+$ and applying integration by parts we have
 \begin{equation}\label{a1}
 \int_{O_+}|\nabla\left(\alpha_0-\alpha_{+}\right)|^2d\rho dz=-\int_{\Sigma/U(1)}\left(\alpha_0-\alpha_{+}\right)\partial_r\left(\alpha_0-\alpha_{+}\right).
 \end{equation}Similarly, we have $\Delta_{\rho,z}\left(\alpha_0-\alpha_{+}\right)=0$ on ${\Omega}$ which yields to
  \begin{equation}\label{a2}
 \int_{O_-}|\nabla\left(\alpha_0-\alpha_{-}\right)|^2d\rho dz=\int_{\Sigma/U(1)}\left(\alpha_0-\alpha_{-}\right)\partial_r\left(\alpha_0-\alpha_{-}\right).
 \end{equation}Putting this together with equation \eqref{mean} yield to the following expression
 \begin{equation}
 \begin{split}
 \int_{O_+}|\nabla\left(\alpha_0-\alpha_{+}\right)|^2d\rho dz+&\int_{O_-}|\nabla\left(\alpha_0-\alpha_{-}\right)|^2d\rho dz\\
 &=\int_{\Sigma/U(1)}\left(\alpha_0-\alpha\right)\left(H_+(\Sigma,\bar{g}_+)-H_-(\Sigma,\bar{g}_-)\right)e^{-U_0+\alpha},
 \end{split}
 \end{equation}where we used $\alpha=\alpha_-=\alpha_+$ on $\Sigma$. Since the mean curvature is the same on $\Sigma$, we have $\nabla\left(\alpha_0-\alpha_{+}\right)=0$ on $O_+$ and $|\nabla\left(\alpha_0-\alpha_{-}\right)|=0$ on $O_-$. Thus $\alpha_0-\alpha$ is a constant away from $\Sigma$. In light of vanishing boundary condition at axis to avoid conical singularity, we have $\alpha=\alpha_0$ away from $\Sigma$. Moreover, since $\alpha$ is continuous across $\Sigma$, $\alpha=\alpha_0$ on $\tilde{M}^3$. Then $(\bar{\mathcal{G}},\bar{\mathcal{K}},\bar{\mathcal{E}},\bar{\mathcal{B}})$ is isometric to $(g_0,k_0,E_0,B_0)$ and $\bar{\mu}=0$ on whole $\tilde{M}^3$. Therefore, by equation \eqref{bmu}, we have
\begin{equation}
 \hat{B}=\hat{E}=\pi=\mu=0,\qquad \partial_{\rho}A_{ z}=\partial_zA_{\rho} 
 \end{equation}It follows that 1-form $A_{\rho}d\rho+A_{z}dz$ is closed on $\mathbb{R}^3$, so there exist a potential $f$ such that $\partial_{\rho}f=A_{\rho}$ and $\partial_{z}f=A_{z}$. Thus under change of coordinates $\tilde{\phi}=\phi^i+f(\rho,z)$, the metric takes the form
 \begin{equation}
 g=e^{-2U_0+2\alpha_0}\left(d\rho^2+dz^2\right)+\rho e^{-2U_0}d\tilde{\phi}^2,
 \end{equation} which implies $({\mathcal{G}},{\mathcal{K}},{\mathcal{E}},{\mathcal{B}})=(\bar{\mathcal{G}},\bar{\mathcal{K}},\bar{\mathcal{E}},\bar{\mathcal{B}})$ is isometric to canonical slice of extreme Kerr-Newman spacetime.
	\end{proof}
Next we use Theorem \ref{thm1.2} and prove the Corollary \ref{cor1}.
\begin{proof}[Proof of Corollary \ref{cor1}] Since $\partial M^n$ has non-negative mean curvature and it is boundary of an asymptotically flat manifold, by barrier argument there exist an outermost minimal surface $\Sigma_{\text{min}}$ in $M^n$ and by \cite[Proposition 3.1.]{bryden2018positive}, $\Sigma_{\text{min}}$ is axially symmetric.  Assume $W\subset M^n$ such that $\partial W=\partial M^n\cup \Sigma_{\text{min}}$. In three dimensions using \cite[Lemma 4.1]{huisken2001inverse}, we have $M^3\,\backslash W$ is diffeomorphic to the complement of a finite number of open $3$-balls in $\mathbb{R}^3$ and it is simply connected. 
	
In four dimensions, first we show $M^4\,\backslash W$ is simply connected. Assume $\pi_1(M^4\backslash W)\neq 0$, then by residual finiteness, it admits a finite nontrivial universal cover. Thus by asymptotically flatness it must have finite number of ends which by barrier argument for the mean curvature it includes a minimal surface. Clearly this minimal surface is not contain entirely in the boundary and so its projection violate outermost minimal surface condition for $\Sigma_{\text{min}}$ and we get a contradiction. Therefore, $M^4\,\backslash W$ is simply connected. Moreover, since it has isometry subgroup $U(1)^2$ and spin by Orlik and Raymond \cite{orlik1970actions} and Hollands and Ishibashi \cite{hollands2011further}, it should be diffeomorphic to $\mathbb{R}^4\#l(S^2\times S^2)$, for $l\in\mathbb{N}$, minus $4$-manifold with boundaries $S^1\times S^2$ or $S^3$ (its quotients).
	
(a) Since $\Sigma_{\text{min}}$ has only one component, $M^3\,\backslash W$ is diffeomorphic to $\mathbb{R}^3$ minus a 3-ball. Then we deform it by Proposition \ref{propmain} and Proposition \ref{Prop3.1} and we get a $C^2$ metric with $C^1$ second fundamental form and electromagnetic fields. Following proof of Theorem \ref{thm1.2}-(a) we get the strict inequality.

(b) If $M^4$ is spin, $\pi_1(\Sigma_{\text{min}})=0$, and $H_2(M^4\backslash W)=0$, we get $M^4\,\backslash W$ is diffeomorphic to $R^4$ minus a 4-ball. Following proof of Theorem \ref{thm1.2}-(b) we get the strict inequality.

(c) If $M^4$ is spin, $\mathcal{E}=\mathcal{B}=0$, $\pi_1(\Sigma_{\text{min}})=\mathbb{Z}$, and $H_2(M^4\,\backslash W)=\mathbb{Z}$, we get $M^4\,\backslash W$ is diffeomorphic to $\mathbb{R}^4\#(S^2\times B^2)$. Then we deform it by Proposition \ref{propmain} and Proposition \ref{Prop3.1} and we get a $C^2$ metric with $C^1$ second fundamental form. This is similar to initial data of non-extreme black ring and we can apply \cite{alaee2017mass} and get the result. 	\end{proof}

\paragraph*{\emph{Acknowledgements.}} We would like to thank Marcus Khuri for comments on the first draft of this paper. Aghil Alaee acknowledges the support of NSERC Postdoctoral Fellowship, the Gordon and Betty Moore Foundation, and the John Templeton Foundation. Shing-Tung Yau acknowledges the support of NSF Grant DMS-1607871.

\bibliographystyle{abbrv}
%\bibliography{masterfile} 

\begin{thebibliography}{}

\end{thebibliography}


\begin{thebibliography}{10}
	
	\bibitem{alaee2015proof}
	A.~Alaee, M.~Khuri, and H.~Kunduri.
	\newblock Proof of the mass-angular momentum inequality for bi-axisymmetric
	black holes with spherical topology.
	\newblock {\em Advances in Theoretical and Mathematical Physics},
	20:1397--1441, 2016.
	
	\bibitem{alaee2017mass}
	A.~Alaee, M.~Khuri, and H.~Kunduri.
	\newblock Mass--angular-momentum inequality for black ring spacetimes.
	\newblock {\em Physical review letters}, 119(7):071101, 2017.
	
	\bibitem{alaee2017relating}
	A.~Alaee, M.~Khuri, and H.~Kunduri.
	\newblock Relating mass to angular momentum and charge in five-dimensional
	minimal supergravity.
	\newblock {\em Annales Henri Poincar{\'e}}, volume~18, pages 1703--1753.
	Springer, 2017.
	
	\bibitem{alaee2019existence}
	A.~Alaee, M.~Khuri, and H.~Kunduri.
	\newblock Existence and uniqueness of stationary solutions in 5-dimensional
	minimal supergravity.
	\newblock {\em arXiv preprint arXiv:1904.12425}, 2019.
	
	\bibitem{alaee2019geometric}
	A.~Alaee, M.~Khuri, and S.-T. Yau.
	\newblock Geometric inequalities for quasi-local masses.
	\newblock {\em arXiv preprint arXiv:1910.07081}, 2019.
	
	\bibitem{alaee2014small}
	A.~Alaee and H.~K. Kunduri.
	\newblock Small deformations of extreme five dimensional myers-perry black hole
	initial data.
	\newblock {\em General Relativity and Gravitation}, 47(2):1--29, 2014.
	
	\bibitem{alaee2015remarks}
	A.~Alaee and H.~K. Kunduri.
	\newblock Remarks on mass and angular momenta for {$U (1)^2$}-invariant initial
	data.
	\newblock {\em Journal of Mathematical Physics}, 57(3):032502, 2016.
	
	\bibitem{alaee2014notes}
	A.~Alaee, H.~K. Kunduri, and E.~M. Pedroza.
	\newblock Notes on maximal slices of five-dimensional black holes.
	\newblock {\em Classical and Quantum Gravity}, 31(5):055004, 2014.
	
	\bibitem{alaee2019localized}
	A.~Alaee, M.~Lesourd, and S.-T. Yau.
	\newblock A localized spacetime penrose inequality and horizon detection with
	quasi-local mass.
	\newblock {\em arXiv preprint arXiv:1912.01581}, 2019.
	
	\bibitem{bray2001proof}
	H.~L. Bray.
	\newblock Proof of the riemannian penrose inequality using the positive mass
	theorem.
	\newblock {\em Journal of Differential Geometry}, 59(2):177--267, 2001.
	
	\bibitem{bryden2018positive}
	E.~T. Bryden, M.~A. Khuri, and B.~D. Sokolowsky.
	\newblock The positive mass theorem with angular momentum and charge for
	manifolds with boundary.
	\newblock {\em Journal of Mathematical Physics}, 60(5):052501, 2019.
	
	\bibitem{cha2015deformations}
	Y.~S. Cha and M.~A. Khuri.
	\newblock Deformations of charged axially symmetric initial data and the
	mass--angular momentum--charge inequality.
	\newblock In {\em Annales Henri Poincar{\'e}}, volume~16, pages 2881--2918.
	Springer, 2015.
	
	\bibitem{chrusciel2008mass}
	P.~T. Chru{\'s}ciel.
	\newblock Mass and angular-momentum inequalities for axi-symmetric initial data
	sets i. positivity of mass.
	\newblock {\em Annals of Physics}, 323(10):2566--2590, 2008.
	
	\bibitem{chrusciel2008mass1}
	P.~T. Chru{\'s}ciel, Y.~Li, and G.~Weinstein.
	\newblock Mass and angular-momentum inequalities for axi-symmetric initial data
	sets. ii. angular momentum.
	\newblock {\em Annals of Physics}, 323(10):2591--2613, 2008.
	
	\bibitem{costa2010proof}
	J.~L. Costa.
	\newblock Proof of a dain inequality with charge.
	\newblock {\em Journal of Physics A: Mathematical and Theoretical},
	43(28):285202, 2010.
	
	\bibitem{dain2012geometric}
	S.~Dain.
	\newblock Geometric inequalities for axially symmetric black holes.
	\newblock {\em Classical and Quantum Gravity}, 29(7):073001, 2012.
	
	\bibitem{dain2008proof}
	S.~Dain.
	\newblock Proof of the angular momentum-mass inequality for axisymmetric black
	holes.
	\newblock {\em Journal of Differential Geometry}, 79:33--67, 2008.
	
	\bibitem{dain2013lower}
	S.~Dain, M.~Khuri, G.~Weinstein, and S.~Yamada.
	\newblock Lower bounds for the area of black holes in terms of mass, charge,
	and angular momentum.
	\newblock {\em Physical Review D}, 88(2):024048, 2013.
	
	\bibitem{figueras2011black}
	P.~Figueras, K.~Murata, and H.~S. Reall.
	\newblock Black hole instabilities and local penrose inequalities.
	\newblock {\em Classical and Quantum Gravity}, 28(22):225030, 2011.
	
	\bibitem{friedman1993topological}
	J.~L. Friedman, K.~Schleich, and D.~M. Witt.
	\newblock Topological censorship.
	\newblock {\em Physical Review Letters}, 71(10):1486, 1993.
	
	\bibitem{galloway1995topology}
	G.~J. Galloway.
	\newblock On the topology of the domain of outer communication.
	\newblock {\em Classical and Quantum Gravity}, 12(10):L99, 1995.
	
	\bibitem{galloway2006generalization}
	G.~J. Galloway and R.~Schoen.
	\newblock A generalization of hawkingï¿½s black hole topology theorem to higher
	dimensions.
	\newblock {\em Communications in Mathematical Physics}, 266(2):571--576, 2006.
	
	\bibitem{gibbons1972time}
	G.~W. Gibbons.
	\newblock The time symmetric initial value problem for black holes.
	\newblock {\em Communications in Mathematical Physics}, 27(2):87--102, 1972.
	
	\bibitem{hollands2011further}
	S.~Hollands, J.~Holland, and A.~Ishibashi.
	\newblock Further restrictions on the topology of stationary black holes in
	five dimensions.
	\newblock In {\em Annales Henri Poincare}, volume~12, pages 279--301. Springer,
	2011.
	
	\bibitem{hollands2008uniqueness}
	S.~Hollands and S.~Yazadjiev.
	\newblock Uniqueness theorem for 5-dimensional black holes with two axial
	killing fields.
	\newblock {\em Communications in Mathematical Physics}, 283(3):749--768, 2008.
	
	\bibitem{huisken2001inverse}
	G.~Huisken, T.~Ilmanen, et~al.
	\newblock The inverse mean curvature flow and the riemannian penrose
	inequality.
	\newblock {\em Journal of Differential Geometry}, 59(3):353--437, 2001.
	
	\bibitem{khuri2018plumbing}
	M.~Khuri, Y.~Matsumoto, G.~Weinstein, and S.~Yamada.
	\newblock Plumbing constructions and the domain of outer communication for
	5-dimensional stationary black holes.
	\newblock {\em arXiv preprint arXiv:1807.03452}, 2018.
	
	\bibitem{khuri2016positive}
	M.~Khuri and G.~Weinstein.
	\newblock The positive mass theorem for multiple rotating charged black holes.
	\newblock {\em Calculus of Variations and Partial Differential Equations},
	55(2):33, 2016.
	
	\bibitem{lee2013positive}
	D.~Lee.
	\newblock A positive mass theorem for lipschitz metrics with small singular
	sets.
	\newblock {\em Proceedings of the American Mathematical Society},
	141(11):3997--4004, 2013.
	
	\bibitem{lee2015positive}
	D.~A. Lee and P.~G. LeFloch.
	\newblock The positive mass theorem for manifolds with distributional
	curvature.
	\newblock {\em Communications in Mathematical Physics}, 339(1):99--120, 2015.
	
	\bibitem{liu2006positivity}
	C.-C.~M. Liu and S.-T. Yau.
	\newblock Positivity of quasi-local mass ii.
	\newblock {\em Journal of the American Mathematical Society}, 19(1):181--204,
	2006.
	
	\bibitem{miao2002positive}
	P.~Miao.
	\newblock Positive mass theorem on manifolds admitting corners along a
	hypersurface.
	\newblock {\em Advances in Theoretical and Mathematical Physics},
	6(6):1163--1182, 2002.
	
	\bibitem{orlik1970actions}
	P.~Orlik and F.~Raymond.
	\newblock Actions of the torus on 4-manifolds. i.
	\newblock {\em Transactions of the American Mathematical Society},
	152(2):531--559, 1970.
	
	\bibitem{schoen1979proof}
	R.~Schoen and S.-T. Yau.
	\newblock On the proof of the positive mass conjecture in general relativity.
	\newblock {\em Communications in Mathematical Physics}, 65(1):45--76, 1979.
	
	\bibitem{schoen1981proof}
	R.~Schoen and S.-T. Yau.
	\newblock Proof of the positive mass theorem. ii.
	\newblock {\em Communications in Mathematical Physics}, 79(2):231--260, 1981.
	
	\bibitem{schoen2013convexity}
	R.~Schoen and X.~Zhou.
	\newblock Convexity of reduced energy and mass angular momentum inequalities.
	\newblock In {\em Annales Henri Poincar{\'e}}, volume~14, pages 1747--1773.
	Springer, 2013.
	
	\bibitem{shi2002positive}
	Y.~Shi, L.-F. Tam.
	\newblock Positive mass theorem and the boundary behaviors of compact manifolds
	with nonnegative scalar curvature.
	\newblock {\em Journal of Differential Geometry}, 62(1):79--125, 2002.
	
	\bibitem{shibuya2018lorentzian}
	K.~Shibuya.
	\newblock Lorentzian positive mass theorem for spacetimes with distributional
	curvature.
	\newblock {\em arXiv preprint arXiv:1803.10387}, 2018.
	
	\bibitem{wald2010general}
	R.~M. Wald.
	\newblock {\em General relativity}.
	\newblock University of Chicago press, 2010.
	
	\bibitem{zhou2014}
	X.~Zhou.
	\newblock Mass angular momentum inequality for axisymmetric vacuum data with
	small trace.
	\newblock {\em Communications in Analysis and Geometry}, 22(3):519 -- 571,
	2014.
	
\end{thebibliography}

\end{document}